\documentclass{amsart}%

\usepackage{amssymb}
\usepackage{amsmath}
\usepackage{amsthm}
\usepackage{amscd,mdwlist}
\usepackage[margin=1in]{geometry}
\usepackage[utf8]{inputenc}
\usepackage[colorlinks=true, linkcolor=purple, citecolor=purple]{hyperref}
\usepackage{xcolor}
\usepackage{cite}
\usepackage{float}
\usepackage{graphicx}
\usepackage{dsfont} 
\usepackage{enumerate}
\usepackage{enumitem}
\usepackage{thmtools} 
\usepackage[nameinlink]{cleveref}
\usepackage{esint}
\usepackage{mathtools}
\usepackage{csquotes}
\usepackage{tikz}
\usepackage[hang,flushmargin]{footmisc} 

\theoremstyle{plain}
\newtheorem{theorem}{Theorem}[section]
\newtheorem{proposition}[theorem]{Proposition}
\newtheorem{corollary}[theorem]{Corollary}
\newtheorem{lemma}[theorem]{Lemma}

\theoremstyle{definition}
\newtheorem{definition}[theorem]{Definition}

\theoremstyle{remark}
\newtheorem{remark}[theorem]{Remark}
\newtheorem{examples}[theorem]{Examples}

\numberwithin{equation}{section}

\DeclareMathOperator{\dist}{dist}
\DeclareMathOperator{\supp}{supp}

\DeclareMathOperator{\lip}{Lip}

\DeclareMathOperator{\diverg}{div}
\renewcommand{\div}{\diverg}

\DeclareMathOperator{\re}{Re}
\DeclareMathOperator{\const}{const}

\DeclareMathOperator{\capacity}{Cap}

\def\Xint#1{\mathchoice
	{\XXint\displaystyle\textstyle{#1}}%
	{\XXint\textstyle\scriptstyle{#1}}%
	{\XXint\scriptstyle\scriptscriptstyle{#1}}%
	{\XXint\scriptscriptstyle\scriptscriptstyle{#1}}%
	\!\int}
\def\XXint#1#2#3{{\setbox0=\hbox{$#1{#2#3}{\int}$}
		\vcenter{\hbox{$#2#3$}}\kern-.5\wd0}}

\def\dashint{\Xint-}

\newcommand{\norm}[1]{\left\| #1 \right\|}

\begin{document}

\title{Optimal H\"{o}lder regularity for solutions to Signorini-type obstacle problems}

\author{Ki-Ahm Lee$^1$}
\address{$^1$Department of Mathematical Sciences and Research Institute of Mathematics,
	Seoul National University, Seoul 08826, Republic of Korea}
\email{kiahm@snu.ac.kr}

\author{Se-Chan Lee$^2$}
\address{$^2$School of Mathematics, Korea Institute for Advanced Study, Seoul 02455, Republic of Korea}
\email{sechan@kias.re.kr}

\author{Waldemar Schefer$^3$}
\address{$^3$Fakult\"at f\"ur Mathematik, Universit\"at Bielefeld, Postfach 100131, 33501 Bielefeld, Germany}
\email{wschefer@math.uni-bielefeld.de}


\date{\today}

\begin{abstract}
We study the existence, uniqueness, and regularity of weak solutions to a class of obstacle problems, where the obstacle condition can be imposed on a subset of the domain. In particular, we establish the optimal H\"older regularity for Signorini-type problems, that is, the obstacle condition is imposed only on a subset of codimension one. For this purpose, we employ capacities, Alt--Caffarelli--Friedman-type and Almgren-type monotonicity formulae, and investigate an associated mixed boundary value problem. Further, we apply this problem to study classical obstacle problems for irregular obstacles.
\end{abstract}

\keywords{Obstacle problems; Regularity theory; Monotonicity formula; Mixed boundary value problem}
\subjclass[2020]{35B65, 35J86, 35R35}

\maketitle
\tableofcontents

\section{Introduction}
In this article we are interested in the optimal H\"older regularity of solutions to a class of obstacle problems characterized by the Euler-Lagrange equations
\begin{equation}\label{eq:intro euler lagrange}
	\left\{\begin{aligned}
		-\Delta u &\geq 0 &&\text{in $\Omega$} \\
		-\Delta u &= 0 &&\text{in $\Omega\setminus \{u = \psi \}$} \\
		u &\geq \psi &&\text{on $F$}
	\end{aligned}\right.
\end{equation}
for open domains $\Omega\subset\mathbb R^n$, $n\geq 2$, relatively closed subsets $F\subset \Omega$, and obstacles $\psi$ defined on $F$. 

In the case that $F=\Omega$, \eqref{eq:intro euler lagrange} coincides with the \emph{classical obstacle problem}. In the case that $F= \mathcal M$, for a manifold $\mathcal M$ of codimension one that separates $\Omega$ into two parts, \eqref{eq:intro euler lagrange} coincides with the \emph{thin obstacle problem}, also called \emph{Signorini problem}. There is a vast literature in these two cases as they relate to various problems in applied sciences as physics, biology, mathematical finance, and in pure mathematics within the study of free boundary problems and variational inequalities. We refer to the comprehensive books \cite{DL76, KS80, Fri82, PSU12} and the survey article \cite{FR22} for motivations and applications.

Obstacle problems for more general subsets $F$ appear naturally in the study of \emph{capacitor (condenser) potentials}, see \cite{Lan72}, as the condenser potential of the pair $(\partial\Omega, F)$ solves \eqref{eq:intro euler lagrange}, with respect to the obstacle $\psi = 1$, and vanishes at the boundary. Another motivation comes from \emph{stochastic control theory}, see \cite{BL74, BL75}, which gives rise to obstacle problems with \emph{irregular} obstacles that are neither continuous nor of finite energy. Our main results show the optimal regularity in the case of \emph{Signorini-type problems}, that is, $F$ is contained in a manifold $\mathcal M$ of codimension one. Moreover, we provide a relation between \eqref{eq:intro euler lagrange} with sufficiently regular obstacles and (classical) obstacle problems with discontinuous obstacles.\newline

\subsection{Main results}
Let $\Omega = B_1(0)\subset \mathbb R^n$ be the unit ball centered at the origin, $F$ be a subset of the horizontal hyperplane $\{x_n = 0\}$, and $u \in H^1(\Omega)$ be a weak solution to \eqref{eq:intro euler lagrange}. The regularity of the obstacle $\psi$ in \eqref{eq:intro euler lagrange} will be suitably chosen depending on the context of each theorem. Then, our main results can be summarized as follows:
\begin{enumerate}
	\item \Cref{thm:optimalreg}: $u\in C^{1\slash 2}$ if the \emph{contact set} $\Lambda(u)\coloneqq \{u = \psi\}$ satisfies a capacity density condition.
	\item \Cref{thm-almostoptimal}: $u\in C^{1\slash 2-}$ if $F$ equals half of the horizontal hyperplane.
	\item \Cref{thm:optimal holder reg half line}: $u\in C^{1\slash 2}$ if $n=2$ and $F$ equals half of the horizontal line.
\end{enumerate}
Here, $u \in C^{1/2-}$ means that $u \in C^{1/2-\varepsilon}$ for any $\varepsilon \in (0, 1/2)$. Due to a simple example, which goes back to \cite{Sha68}, the $C^{1\slash 2}$-regularity is optimal for solutions to Signorini-type problems. See the discussion provided at the end of \Cref{sec-generalobstacle}. Let us briefly explain the main difference among the results. The benefit of \Cref{thm:optimalreg} is that no additional assumption on $F$ is required, although it can be difficult to verify the condition in \Cref{thm:optimalreg} due to the a priori unknown contact set. However, it is noteworthy that various types of sets satisfy a capacity density condition, see \Cref{ex:cdc and outer ball}. Further remarks on the contrary case are also provided at the beginning of \Cref{sec-optimalholder}. In \Cref{thm-almostoptimal} and \Cref{thm:optimal holder reg half line} we restrict our attention to the special case when $F$ is half of the horizontal hyperplane. While the approach in \Cref{thm-almostoptimal} allows us to discuss arbitrary dimensions, it only provides the almost optimal regularity. In \Cref{thm:optimal holder reg half line} we have to limit ourselves to the two dimensional setting. For a detailed explanation see the beginning of \Cref{sec:opt holder reg} and \Cref{rem:rellich}. Within the two dimensional setting, we additionally classify the blowup profiles of $u$ at a free boundary point in \Cref{lem-classification}. In polar coordinates, these are given, up to a multiplicative constant, by
\begin{equation*}
    u_0(r, \theta) = r^{\kappa} \cos(\kappa\theta) \quad\text{for some $\kappa \in 2\mathbb N$ or $\kappa \in \mathbb N_0 + \frac 12$.}
\end{equation*}
Furthermore, as we discuss in \Cref{sec:application}, we relate the regularity of solutions to \eqref{eq:intro euler lagrange} to the regularity of solutions to classical or thin obstacle problems with discontinuous obstacles. More precisely, in \Cref{thm:disc obstacle} we show that if $\psi$ is an obstacle defined in $\Omega$ or on $\mathcal M$ that can be separated into two obstacles $\psi_1$ and $\psi_2$ defined on a partition $F_1$ and $F_2$ of $\Omega$ or $\mathcal M$ and $\psi_1$ is strictly larger than $\psi_2$ in a small neighborhood of $x_0\in \overline{F_1}\cap \overline{F_2}$ (meaning $\psi$ has a jump-type discontinuity at $x_0$), then the solution to the obstacle problem related to $\psi$ coincides with the solution to \eqref{eq:intro euler lagrange} with respect to $\psi_1$ and $F_1$ in a neighbourhood of $x_0$. In combination with the regularity results above, this approach shows the optimal regularity for classical and thin obstacle problems with respect to such jump-type discontinuous obstacles. Note, although the regularity results above discuss the zero obstacle case, we provide generalizations in \Cref{thm:optimalreg non zero}, \Cref{thm:almost optimal non zero}, and \Cref{rem:acf non zero}. \newline

Let us mention the literature closely related to our results. The well-posedness of \eqref{eq:intro euler lagrange} in terms of continuous solutions vanishing on the boundary was studied in \cite{LS69} for $C^1$-regular obstacles. Higher regularity of the solution was investigated in \cite{Lew68, Kin71} for a special case in two dimensions where $F$ was assumed to be a straight line segment. More precisely, the authors proved Lipschitz regularity for solutions $u$ that vanish at the boundary and the obstacle $\psi$ belongs to $C^{1, \alpha}$, is nonnegative, and vanishes at the endpoints of the line $F$. One can observe that these structural conditions on $u$ and $\psi$ together with the strong minimum principle imply that the endpoint of $F$ cannot be a free boundary point, unless $u$ is a (trivial) zero solution. Indeed, the setting considered in \cite{Lew68, Kin71} is equivalent to a thin obstacle problem and therefore enables them to achieve higher regularity than $C^{1\slash 2}$, see \Cref{lem:kin71} and the comment afterwards. A crucial difference between this result and ours is the fact that we consider arbitrary boundary values without specific relation between $\psi$ and $u|_{\partial \Omega}$, which allows us only to retrieve $C^{1\slash 2}$-regularity. Since the $C^{1\slash 2}$-regularity is optimal, our results \Cref{thm:optimalreg}, \Cref{thm-almostoptimal}, and \Cref{thm:optimal holder reg half line} generalize the results in \cite{Lew68, LS69, Kin71} in this regard.

In the articles \cite{FM79, FM82, FM84, Mos86, Mos87}, the authors studied classical obstacle problems for irregular obstacles, with varying irregularities. \Cref{thm:disc obstacle} provides a new way to study classical obstacle problems for irregular obstacles having jump-type discontinuities. In particular, the combination with our previous regularity results, give the optimal H\"{o}lder regularity of solutions for such cases. Let us also mention that the approach in \Cref{sec:application} is similar to some results in \cite{Ior83, JOS08}, although here the authors required higher regularity of the obstacle sets.\newline

\subsection{Historical background} 
The classical and thin obstacle problem are quite well understood. Let us provide a brief overview.  The classical obstacle problem originated from \cite{Sta64} while the Signorini problem appeared earlier in \cite{Sig59} and \cite{Fic64}. Since then, various authors contributed to these problems. The first systematic treatment of the well-posedness (in weak sense) was done in \cite{LS67} within the context of variational inequalities, which includes both cases. The optimal regularity of solutions to the classical obstacle problem is $C^{1, 1}$, see \cite{Fre72}. For the thin obstacle problem, the optimal regularity is Lipschitz continuity and $C^{1, 1\slash 2}$ from both sides up to the obstacle,  see \cite{Fre75, Fre77, Ric78, AC04}. In subsequent studies, the regularity and structure of the \emph{free boundary} $\partial\{u > \psi \}$ were analyzed. In the classical case, the combination of results in \cite{Caf77, KN77, CR77, Caf98, Mon03} implies that the free boundary is smooth up to a set of singular points. The local character of free boundary points was characterized in \cite{Caf98} by \emph{blowups}, saying: Either a free boundary point $x_0$ is \emph{regular}, then its blowup is given by the function $x\mapsto (x_1)_+^2$ (up to translation and rotation), or $x_0$ is \emph{singular}, then its blowup is given by the function $x \mapsto x_1^2$ (up to translation and rotation). In the thin case, the problem is more involved. The first result in this direction is contained in \cite{Lew72}, which shows in the two-dimensional case that the free boundary is a finite amount of intervals provided that the obstacle is real analytic. Another result in this direction was established in \cite{ACS08}, showing that the set of regular points in the free boundary is $C^{1, \alpha}$-regular. Combined with the bootstrap argument in \cite{KPS15}, this further implies smoothness of the set of regular points. In contrast to the classical case, the set of regular points may be empty and the set of non-regular points may be large. See also the survey \cite{FR22}. We note that there are many recent articles related to the finer structure of the free boundary that we do not mention. \newline

In the aforementioned articles on the classical and thin obstacle problems, the usual assumption on the obstacle $\psi$ is to be either zero or to be sufficiently regular in a pointwise sense. On the other hand, motivated by the relation to potential theory and stochastic control theory mentioned above, another line of research was dedicated to the study of (classical) obstacle problems for \emph{irregular obstacles}. One of the first articles in this direction was \cite{LS71}, where the authors showed that $C^{1, \alpha}$ and Lipschitz regularity of the obstacle translates into the same regularity of the solution. Such pointwise regularity results were extended to finer moduli of continuity in \cite{CK80} and to Dini-continuity of the gradient in \cite{Ok17}. These results intuitively say \enquote{the solution is as regular as the obstacle, up to the optimal regularity}. Such results were also achieved in the context of Morrey spaces, \cite{Cho91}, and Campanato spaces, \cite{Ele07}. Articles directly influenced by potential theory and stochastic control theory were \cite{FM79, FM82, FM84, Mos86, Mos87}. They investigated pointwise regularity of solutions for \emph{discontinuous} obstacles which are not contained in a Sobolev space. More precisely, they assumed the obstacle to satisfy a \emph{one-sided H\"{o}lder condition} or a \emph{unilateral regularity condition of Wiener type} and showed continuity, H\"{o}lder continuity, and Wiener-modulus of continuity of solutions.

A similar line of research is focused in Calder\'{o}n-Zygmund type estimates for classical obstacle problems, which roughly means that Sobolev regularity of the obstacle translates into the same Sobolev regularity of the solution. See, for example, \cite{EH10, BDM11, BCO16}. Further articles in this area generalized the treated operators and associated energy functionals, the underlying function spaces, and domain $\Omega$. See, for example, \cite{Gri21, TNH23, BC24} and the references therein. Many of such results are extended to double obstacle, multi-phase, and parabolic problems.\newline

Let us return to the original problem \eqref{eq:intro euler lagrange} for general subsets $F$ of $\Omega$ and try to give a complete overview. The first article related to this problem is \cite{Lew68}. Here, the author studied the existence of continuous solutions and the free boundary in a two-dimensional setting with $F$ being a closed straight line segment. The same problem, for an elliptic operator of divergence type, was treated in \cite{Kin71}. The author proved Lipschitz regularity of solutions if the obstacle is $C^{1, \alpha}(F)$ and vanishes at the endpoints of the line segment $F$. A similar problem was considered in \cite{Giu71} for higher dimensions. In the follow-up article \cite{LS69} to \cite{Lew68}, the authors extended the existence of continuous solutions to the case of arbitrary closed subsets $F$ of positive capacity in any dimension. In \cite{SV72} the authors proved Lipschitz regularity in the case when $C$ is a uniformly convex set and $\psi$ a smooth obstacle on $F = \Omega\setminus C$. The articles \cite{Ior72, Ior82} treat double obstacle problems for obstacles $\varphi$ and $\psi$ defined on subsets $E$ and $F$, respectively, and the solution $u$ satisfies $u\leq \varphi$ on $E$ and $u\geq \psi$ on $F$. The author proved $\alpha$-H\"{o}lder continuity of $u$ for $\alpha = 1- n\slash p$ if $\varphi, \psi \in W^{1, p}(\Omega)$, $p>n$, and the sets $\Omega\setminus E$ and $\Omega\setminus F$ are Lipschitz domains. The results were extended in the follow-up paper \cite{Ior83} to include finitely many obstacles $\varphi_i$ and $\psi_i$ on subsets $E_i$ and $F_i$. A decade later, in the article \cite{BS93} and the following articles, the authors called \eqref{eq:intro euler lagrange} for subsets $F\subset\Omega$ \enquote{inner obstacle problems}. In \cite{BS93} the authors showed $C^{1, 1}$-regularity of solutions, under the assumptions that $F$ and $\psi$ are sufficiently regular and satisfy an \enquote{egg-shape} condition. In \cite{JOS03, JOS08} the authors improved the results from \cite{Ior82, Ior83} to Lipschitz regularity of solutions under the assumptions that obstacles are smooth and obstacle sets are compact subsets of $\Omega$ with smooth boundaries. A similar result appeared in the textbook \cite[p.\ 137 - 139]{KS80}. The article \cite{JOS05} dealt with the convergence of solutions if there is a sequence of obstacle sets $F_n$ and obstacles $\psi_n$ that convergence in an appropriate sense. In \cite{JOS09} the authors studied $W^{2, p}$-regularity of solutions if $\Omega$ and $F$ are \enquote{strictly star-shaped} and the obstacle has a concave extension to all of $\Omega$.\newline

\subsection{Strategy of the proofs}
Let us briefly illustrate the main strategies to prove \Cref{thm:optimalreg}, \Cref{thm-almostoptimal}, and \Cref{thm:optimal holder reg half line}. In the proof of \Cref{thm:optimalreg} the key tool is the monotonicity formula in \Cref{lem:acf frequency} for an Alt--Caffarelli--Friedman-type frequency function given by
\begin{equation*}
    \beta(r)=\beta(r, u)\coloneqq \frac{1}{r} \int_{B_r} \frac{|\nabla u|^2}{|x|^{n-2}} \, dx.
\end{equation*}
The major difficulty in our situation is the lack of $C^1$-regularity of $u$, which requires an additional effort to obtain the integral identities involved  in the proof of \Cref{lem:acf frequency}. Then, we conclude that the capacity density condition \eqref{eq:CDC contact set} for the contact set guarantees the optimal regularity of $u$ with the help of Maz'ya's inequality \Cref{lem:mazya}. We refer to \cite{ACF84, CJK02} for the original Alt--Caffarelli--Friedman monotonicity formula and some of its variants. The frequency we use was first introduced in \cite{AC04}, but used differently, see \Cref{rem:acf function}.

The key idea in the proof of the almost optimal regularity result in \Cref{thm-almostoptimal} is to establish a connection between Signorini-type problems and mixed boundary value problems (or mixed BVPs for short). To be precise, if we let $\Psi$ be the solution to the associated mixed BVP, then $u$ can be interpreted as a solution to a classical obstacle problem with obstacle $\Psi$. Note, in the context of thin obstacle problems, this approach is well known by using Dirichlet boundary value problems instead of mixed BVPs. In this view, we first investigate the almost optimal H\"older regularity of solutions to mixed BVPs by constructing appropriate barrier functions, see \Cref{thm-almostoptimal-mixedBVP}. Then, we transport this regularity to the solution $u$ of \eqref{eq:intro euler lagrange} with the help of the regularity result for classical obstacle problems in \cite{Caf98}. We expect that this argument is also applicable to the viscosity framework, including
fully nonlinear operators in non-divergence form, as it does not rely on a weak formulation or integration by parts formulas.

The optimal $C^{1/2}$-regularity in \Cref{thm:optimal holder reg half line} follows from blowup arguments. For this purpose, we present a Rellich-type formula in \Cref{prop:rellich} to prove the monotonicity of Almgren's frequency function
\begin{equation*}
    N(r)=N(r, u) \coloneqq \frac{r \int_{B_r}|\nabla u|^2 \, dx}{\int_{\partial B_r} u^2 \, d\sigma}.
\end{equation*}
In fact, Almgren's monotonicity formula allows us to utilize the blowup analysis for $u$: $(i)$ the rescaling $u_r$ of $u$ converge to a limit $u_0$ in \Cref{prop:uni bound blowup}, $(ii)$ $u_0$ is a global \emph{homogeneous} solution to a Signorini-type problem in \Cref{lem-blowup}, and $(iii)$ $u$ has an optimal regularity coming from the classification of blowups in \Cref{lem-classification}. Again, since we cannot expect $C^1$-regularity of solutions on $F$, there are several delicate difficulties when we follow the standard argument developed in \cite{ACS08} for the thin obstacle problems. Indeed, we would like to point out that the lack of $C^1$-regularity can be overcome by capturing a kind of degenerate effect in dimension two. For details of the difficulty of this approach in higher dimensions, see \Cref{rem:rellich}.\newline

\subsection{Overview}
The paper is organized as follows. In \Cref{sec-generalobstacle}, we introduce the necessary notations, definitions, and preliminary results for general obstacle problems. \Cref{sec-optimalholder}, \Cref{sec-almostoptimal}, and \Cref{sec:opt holder reg} are concerned with the proof of the main results \Cref{thm:optimalreg}, \Cref{thm-almostoptimal}, and \Cref{thm:optimal holder reg half line}, respectively. More precisely, the proofs in these sections rely on the Alt--Caffarelli--Friedman-type monotonicity formula, the analysis of mixed boundary value problems, and Almgren's monotonicity formula, respectively. Possible generalizations to non-zero obstacles are discussed in \Cref{thm:optimalreg non zero}, \Cref{thm:almost optimal non zero}, and \Cref{rem:acf non zero}.  Finally, in \Cref{sec:application}, we apply these results to study obstacle problems for obstacles with jump-type discontinuities.

\subsection*{Acknowledgements}
The third author thanks Florian Grube for valuable discussions regarding some of the technical parts in the proofs of \Cref{sec:opt holder reg}. Ki-Ahm Lee is supported by the National Research Foundation of Korea (NRF) grant funded by the Korea government (MSIP): NRF-2020R1A2C1A01006256. Se-Chan Lee is supported by the KIAS Individual Grant (No. MG099001) at Korea Institute for Advanced Study. Waldemar Schefer is supported by the German Research Foundation (DFG GRK 2235 - 282638148).

\section{The general obstacle problem}\label{sec-generalobstacle}
For any open set $\Omega\subset \mathbb R^n$, $n\geq 2$, let $H^1(\Omega)$ denote the usual Sobolev space defined as the subspace of $L^2(\Omega)$ with weak derivatives in $L^2(\Omega)$ and norm
\begin{equation*}
	\norm{u}_{H^1(\Omega)} = \left(\norm{u}_{L^2(\Omega)}^2 + \norm{\nabla u}_{L^2(\Omega)}^2 \right)^{\frac 12}.
\end{equation*}
Let $H^1_0(\Omega)$ denote the closure of the set of smooth compactly supported functions $C_c^\infty(\Omega)$ in $\Omega$ with respect to the $H^1(\Omega)$-norm. Then, the \emph{relative $1$-capacity} of an open set $U\subset \Omega$ is defined by
\begin{equation*}
	\capacity_1(U;\Omega) \coloneqq  \inf\left\{ \norm{u}_{L^2(\Omega)}^2 + \norm{\nabla u}_{L^2(\Omega)}^2 \mid u \in H^1(\Omega) \text{ such that  $u\geq 1$ a.e.\  on $U$} \right\},
\end{equation*}
where we set $\capacity_1(U;\Omega)$ as $+\infty$ if there is no such $u$. Moreover, for an arbitrary set $A\subset \Omega$ we define
\begin{equation*}
	\capacity_1(A;\Omega) \coloneqq  \inf_{ \substack{A\subset U \subset \Omega \\ U:\text{ open}}} \capacity_1(U;\Omega).
\end{equation*}
From \cite[Section 2.1]{fot}, any $u\in H^1(\Omega)$ has a \emph{quasi-continuous} representative $\widetilde u:\Omega\to\mathbb R$ with respect to $\capacity_1$, which means that for every $\varepsilon>0$ there exists an open set $G\subset \Omega$ such that $\capacity_1(G;\Omega) < \varepsilon$, $\widetilde u\in C(\Omega\setminus G)$, and $\widetilde u = u$ almost everywhere. For any set $F\subset \Omega$, $u\in H^1(\Omega)$, and any extended real-valued function $\psi: F \to \overline{\mathbb R} \coloneqq  [-\infty, +\infty]$ we write $u \geq \psi$ \emph{quasi-everywhere} (\emph{q.e.}) on $F$ if there exists $N\subset F$ such that $\capacity_1(N;\Omega) = 0$ and $\widetilde u \geq \psi$ on $F\setminus N$. In particular, \cite[Theorem 2.1.5]{fot} implies for arbitrary sets $F\subset\Omega$ that
\begin{equation*}
	\capacity_1(F;\Omega) = \inf\left\{\norm{u}_{L^2(\Omega)}^2 + \norm{\nabla u}_{L^2(\Omega)}^2 \mid u \in H^1(\Omega) \text{ such that  $u\geq 1$ q.e.\  on $F$} \right\}.
\end{equation*}
In the present paper, we consider the following \emph{obstacle problems}:
\begin{definition}[General obstacle problems]\label{def:obstacle problem}
	Let $\Omega\subset \mathbb R^n$ be open and $F\subset \Omega$. For boundary data $g\in H^1(\Omega)$ and obstacle $\psi:F\to \overline{\mathbb R}$ define
	\begin{equation*}
		\mathcal K_{g, \psi, F}(\Omega) \coloneqq  \left\{v\in H^1(\Omega) \mid v-g \in H^1_0(\Omega) \text{ and $v\geq \psi$ q.e.\ on $F$} \right\}.
	\end{equation*}
	A function $u:\Omega\to\mathbb R$ is called a \emph{solution to the $\mathcal K_{g, \psi, F}(\Omega)$-obstacle problem} if $u\in \mathcal K_{g, \psi, F}(\Omega)$ and
	\begin{equation*}
		\int_{\Omega} |\nabla u|^2 \, dx = \inf_{v \in \mathcal{K}_{g, \psi, F}(\Omega)} \int_{\Omega} |\nabla v|^2 \, dx.
	\end{equation*}
	We call a solution $u$ \emph{unique} if it is unique in $H^1(\Omega)$. For general $F\subset \mathbb R^n$ let $\mathcal K_{g, \psi, F}(\Omega)$ denote $\mathcal K_{g, \psi, F\cap\Omega}(\Omega)$.
\end{definition}

If $F = \Omega$, then the $\mathcal K_{g, \psi, \Omega}(\Omega)$-obstacle problem coincides with the original, often called \emph{classical} or \emph{thick}, \emph{obstacle problem}. If $F = \mathcal M$ for an $(n-1)$-dimensional manifold $\mathcal M\subset \Omega$ that separates $\Omega$ into two parts, then the $\mathcal K_{g, \psi, \mathcal M}(\Omega)$-obstacle problem coincides with a \emph{thin obstacle problem} or also called \emph{Signorini problem}. In both cases, solutions in the sense of \Cref{def:obstacle problem} are often called \emph{weak} or \emph{variational solutions}. 

Let us discuss some cases that are included in \Cref{def:obstacle problem}.
\begin{enumerate}
	\item Assume that $\capacity_1(F;\Omega) = 0$. Then, every function $v\in H^1(\Omega)$ satisfies $v\geq \psi$ q.e.\ on $F$. Thus, $\mathcal K_{g, \psi, F}(\Omega) = \{v\in H^1(\Omega) \mid v-g \in H^1_0(\Omega) \}$ and therefore the $\mathcal K_{g, \psi, F}(\Omega)$-obstacle problem coincides with the variational formulation of the Dirichlet boundary value problem
	\begin{equation*}
    \left\{\begin{aligned}
			-\Delta u &= 0 &&\text{in $\Omega$} \\
			u &= g &&\text{on $\partial\Omega$.}
		\end{aligned}\right.
	\end{equation*}
	\item Assume that $\capacity_1(\Omega^c;\mathbb R^n) = 0$. Then, \cite[Corollary 2.39, Theorems 2.44 and 2.45]{HKM06} imply that $H^1(\Omega) = H^1_0(\Omega) = H^1(\mathbb R^n)$. Thus, in this case the condition $v-g \in H^1_0(\Omega)$ in \Cref{def:obstacle problem} always holds and so we have
	\begin{equation*}
		\mathcal K_{g, \psi, F}(\Omega) = \{v\in H^1(\mathbb R^n) \mid v\geq \psi \text{ q.e.\ on $F$} \}.
	\end{equation*}
	\item Let $\Psi : \Omega\to\mathbb R$ be an extension of $\psi : F \to \mathbb R$ simply by extending as $-\infty$ on $\Omega\setminus F$. Then, for every $u\in H^1(\Omega)$ we have $u \geq \psi$ q.e.\ on $F$ if and only if $u \geq \Psi$ q.e.\ on $\Omega$. Thus, we have
	\begin{equation*}
		\mathcal K_{g, \psi, F}(\Omega) = \mathcal K_{g, \Psi, \Omega}(\Omega).
	\end{equation*}
	In this sense, \Cref{def:obstacle problem} coincides with the definition of obstacle problems studied in \cite{BB11} and, up to the generality of considered operators, in \cite{FM84, MZ91}.
\end{enumerate}

\begin{remark}
    In the case $\Omega = \mathbb R^n$, $\capacity_1(\cdot ; \mathbb R^n)$ is sometimes referred as \emph{Sobolev capacity}, as for example in \cite{HKM06, BB11}. As we are interested in obstacle problems, we usually assume that $\capacity_1(F;\Omega)>0$ as in the contrary case the obstacle does not have any influence. Note, this implies that $\capacity_1(F;\mathbb R^n) > 0$ due to a simple monotonicity argument. On the other hand, whenever $\Omega$ is an $H^1$-extension domain, the contrary is also true. Thus, for most problems of interest, one can replace $\capacity_1(\cdot;\Omega)$ by the Sobolev capacity. However, for the sake of possible generalizations, we stick to the relative capacity as it provides a more intrinsic perspective.
\end{remark}

Let us also introduce the \emph{relative $0$-capacity} $\capacity_0(F;\Omega)$, for open sets $\Omega$ and arbitrary subsets $F\subset \Omega$, similar to $\capacity_1$ by setting
\begin{equation*}
    \capacity_0(U;\Omega) \coloneqq  \inf \left\{ \norm{\nabla u}_{L^2(\Omega)} \mid u \in H^1_0(\Omega) \text{ such that $u\geq 1$ a.e.\ on $U$}\right\} \quad\text{for $U\subset\Omega$ open.}
\end{equation*}

\subsection{Existence and uniqueness}
The following result is standard and follows from \cite{LS67}:
\begin{theorem}\label{thm:well posedness}
	Let $\Omega\subset\mathbb R^n$ be open and bounded. Suppose that $F\subset \Omega$, $g\in H^1(\Omega)$, and $\psi: F \to \overline{\mathbb R}$. If $\mathcal K_{g, \psi, F}(\Omega)$ is not empty, then there exists a unique solution to the $\mathcal K_{g, \psi, F}(\Omega)$-obstacle problem.
\end{theorem}
Before we proceed, let us collect typical criteria for the nonemptiness of $\mathcal K_{g, \psi, F}(\Omega)$ and therefore the solvability of the corresponding obstacle problem. Adams' criterion from \cite{Ada82} holds in our setting as follows:
\begin{theorem}[Adams' criterion, {\cite[Theorem 7.3]{BB11}}]
	Let $\Omega\subset\mathbb R^n$ be open and bounded. Suppose that $F\subset\Omega$, $g\in H^1(\Omega)$, and $\psi : F \to \overline{\mathbb R}$. Then, $\mathcal K_{g, \psi, F}(\Omega) \neq \emptyset$ if and only if
	\begin{equation*}
		\int_0^\infty t \capacity_{0}\left(\{x \in F \mid \psi(x) - g(x) > t \}; \Omega\right)\, dt < \infty.
	\end{equation*}
\end{theorem}
For classical obstacle problems, there exists an easier criterion in the case that $\psi\in H^1(\Omega)$. It says that $\mathcal K_{g, \psi, \Omega}(\Omega) \neq \emptyset$ if and only if $(g-\psi)_+ \in H^1_0(\Omega)$, see \cite[Proposition 7.4]{BB11}. We can use this criterion if $\psi : F\to\mathbb R$ has an extension $\Psi\in H^1(\Omega)$, that is, $\Psi = \psi$ q.e.\ on $F$, in the following way:
\begin{proposition}\label{prop:K psi f not empty}
	Let $\Omega\subset\mathbb R^n$ be open and bounded. Suppose that $F\subset\Omega$, $g \in H^1(\Omega)$, and $\psi: F \to \overline{\mathbb R}$. If $\psi$ has an extension $\Psi\in H^1(\Omega)$ such that $(\Psi - g)_+ \in H^1_0(\Omega)$, then $\mathcal K_{g, \psi, F}(\Omega) \neq\emptyset$. If $\mathcal K_{g, \psi, F}(\Omega) \neq\emptyset$ and there exists some extension $\Psi\in H^1(\Omega)$ of $\psi$, then there exists an extension $\widetilde \Psi$ of $\psi$ satisfying $(\Psi - g)_+ \in H^1_0(\Omega)$.
\end{proposition}
\begin{proof}
	First, assume that there exists an extension $\Psi\in H^1(\Omega)$ of $\psi$ such that $(\Psi - g)_+ \in H^1_0(\Omega)$. By \cite[Proposition 7.4]{BB11}, we know that $\mathcal K_{g, \Psi, \Omega}(\Omega) \neq\emptyset$. Since for every $f\in \mathcal K_{g, \Psi, \Omega}(\Omega)$ we have $f \geq \Psi = \psi$ q.e.\ on $F$, we conclude that $f\in\mathcal K_{g, \psi, F}(\Omega)$.
	
	Now assume that there exists $f\in\mathcal K_{g, \psi, F}(\Omega)$ and there is an extension $\Psi\in H^1(\Omega)$ of $\psi$. Then, by setting $\widetilde\Psi \coloneqq  \min\{f, \Psi \}\in H^1(\Omega)$, we immediately have $\widetilde\Psi = \psi$ q.e.\ on $F$. At last, $(\Psi-g)_+ \in H^1_0(\Omega)$ follows from the fact that $(f-g)_+ \in H^1_0(\Omega)$ and $0\leq (\Psi - g)_+ \leq (f - g)_+$; see for example \cite[Lemma 1.25 $(ii)$]{HKM06}.
\end{proof}
Moreover, since $F$ is allowed to be a proper subset of $\Omega$, \Cref{prop:K psi f not empty} has a simple application as only the data close to the boundary matters.
\begin{lemma}
	Let $\Omega\subset\mathbb R^n$ be open and bounded. Suppose that $F\subset\Omega$ and $g, \psi \in H^1(\Omega)$.
	\begin{enumerate}
		\item If $\dist(F, \partial\Omega) > 0$, then $\mathcal K_{g, \psi, F}(\Omega)\neq\emptyset$ for every choice $g, \psi\in H^1(\Omega)$.
		\item If $\dist(F, \partial\Omega) = 0$, then $\mathcal K_{g, \psi, F}(\Omega)\neq\emptyset$ if there is a neighbourhood $U$ of $\overline F\cap \partial\Omega$ such that $g\geq \psi$ q.e.\ on $U$.
	\end{enumerate}
\end{lemma}
\begin{proof}
	Clearly, if $\dist(F, \partial\Omega) > 0$, then there exists a cutoff function $\rho\in C_c^\infty(\Omega)$ satisfying $\rho = 1$ on $F$ and $\rho = 0$ outside some neighbourhood $U$ of $F$ such that $\dist(U, \partial\Omega) > 0$. For $\widetilde \psi = \rho\psi + (1-\rho) g$, we have
	\begin{equation*}
		\widetilde \psi\vert_F = \psi \quad\text{and}\quad \widetilde \psi\vert_{\Omega\setminus U} = g.
	\end{equation*}
	In particular, we have $(\widetilde \psi - g)_+ \in H^1_0(\Omega)$ and so \Cref{prop:K psi f not empty} implies $\mathcal K_{g, \psi, F}(\Omega)\neq\emptyset$.
	
	Now, assume we are in the case of $(ii)$. Take a cutoff function $\rho$ such that $\rho = 1$ on $\overline F\setminus U$ and $\rho = 0$ outside a neighbourhood $V$ of $\overline F$. Here we can choose $V$ such that $\dist(V\setminus U, \partial\Omega) > 0$.  Again, by letting $\widetilde \psi = \rho\psi + (1-\rho)g$, it turns out that
	\begin{equation*}
		\widetilde \psi\vert_{F\setminus U} = \psi, \quad \widetilde\psi\vert_{F\cap U} \geq \psi, \quad\text{and}\quad \widetilde\psi\vert_{\Omega\setminus V} = g.
	\end{equation*}
	Moreover, since
	\begin{equation*}
		\widetilde \psi - g =  \rho(\psi - g) \leq 0 \quad\text{on $U$ and on a small neighbourhood of $\partial\Omega$},
	\end{equation*}
	we obtain $(\widetilde \psi - g)_+ \in H^1_0(\Omega)$ and so \Cref{prop:K psi f not empty} implies $\mathcal K_{g, \psi, F}(\Omega)\neq\emptyset$.
\end{proof}

\subsection{Continuity and harmonicity of solutions}
The following can be seen as an extension of some results in \cite{LS69}, where the authors derived the same statements for bounded domains and $\psi \in C^1(\overline{\Omega})$. However, note that \cite{LS69} treats more general elliptic operators in divergence form, which our prove can be easily generalized to.
\begin{theorem}\label{thm:continuity and harmonicity}
	Let $\Omega\subset\mathbb R^n$ be open, $F\subset \Omega$ be relatively closed, and $\psi : F\to \overline{\mathbb R}$ be upper semi-continuous. Assume that $u\in H^1(\Omega)$ is a solution to the $\mathcal K_{u, \psi, F}(\Omega)$-obstacle problem. Then, the following are true:
	\begin{enumerate}
		\item $u$ is continuous in $\Omega$ and satisfies $u\geq \psi$ on $F$.
		\item $u$ is superharmonic in $\Omega$.
		\item $u$ is harmonic in $\Omega\setminus \Lambda(u)$, where $\Lambda(u) \coloneqq  \{x\in F\mid u(x) = \psi(x) \}$ is the \emph{contact set}.
	\end{enumerate}
\end{theorem}
The proof of \Cref{thm:continuity and harmonicity} consists of two steps. In \Cref{prop:superharmonic} we show the superharmonicity of $u$ in $\Omega$ and the harmonicity of $u$ in $\Omega\setminus F$. On the other hand, in \Cref{prop:continuity} we prove the continuity of $u$ in $\Omega$ and the harmonicity of $u$ in $\Omega\setminus \Lambda(u)$.
\begin{proposition}\label{prop:superharmonic}
	Let $\Omega\subset\mathbb R^n$ be open. Suppose that $F\subset \Omega$, $\psi : F \to \overline{\mathbb R}$, and that $u\in H^1(\Omega)$ is a solution to the $\mathcal K_{u, \psi, F}(\Omega)$-obstacle problem. Then, $u$ is harmonic in $\Omega\setminus F$, that is,
	\begin{equation*}
		\int_{\Omega} \nabla u \cdot \nabla \varphi \, dx = 0 \quad\text{for every $\varphi\in H^1_0(\Omega)$ with $\varphi = 0$ q.e.\ on $F$,}
	\end{equation*} 
	and superharmonic in $\Omega$, that is,
	\begin{equation*}
		\int_{\Omega} \nabla u \cdot \nabla \varphi \, dx \geq 0 \quad\text{for every nonnegative $\varphi\in H^1_0(\Omega)$.}
	\end{equation*}
\end{proposition}
\begin{proof}
	We use standard techniques from calculus of variations. First, we prove harmonicity. For $\varphi\in H^1_0(\Omega)$ such that $\varphi = 0$ q.e.\ on $F$ we have $u + t\varphi \in \mathcal K_{u, \psi, F}(\Omega)$ for every $t\in\mathbb R$. Since $u$ is a minimizer of the Dirichlet integral, we have
	\begin{equation*}
		0 = \frac{d}{dt} \int_{\Omega} |\nabla (u + t\varphi)|^2 \, dx \Big\vert_{t=0} = \frac{d}{dt}\int_{\Omega} \left(|\nabla u|^2 + 2t\nabla u \cdot \nabla \varphi + t^2|\nabla \varphi|^2 \right) \, dx \Big\vert_{t=0} = 2\int_{\Omega}\nabla u \cdot \nabla \varphi \, dx.
	\end{equation*}
	Next, assume that $\varphi\in H^1_0(\Omega)$ is nonnegative. Then, $u + t\varphi \in \mathcal K_{u, \psi, F}(\Omega)$ for every $t\geq 0$. Since $u+t\varphi = (1-t)u + t(u+\varphi)$, the convexity of the Dirichlet integral on $\mathcal K_{u, \psi, F}(\Omega)$ implies
	\begin{equation*}
		0 \leq \frac{d}{dt} \int_{\Omega} |\nabla (u + t\varphi)|^2 \, dx \Big\vert_{t=0} =2\int_{\Omega}\nabla u \cdot \nabla \varphi \, dx.
	\end{equation*}
\end{proof}
The harmonicity implies that $u$ is smooth, therefore continuous, away from $F$. To show the continuity along $F$ we proceed similar to the proof of \cite[Theorem 1]{Caf98}.
\begin{proposition}\label{prop:continuity}
	Let $\Omega\subset\mathbb R^n$ be open, $F\subset \Omega$ be relatively closed, and $\psi: F \to \overline{\mathbb R}$. Assume that $u\in H^1(\Omega)$ is a solution to the $\mathcal K_{u, \psi, F}(\Omega)$-obstacle problem. If $\psi$ is upper semi-continuous on $F$, then $u\in C(\Omega)$ and $u$ is harmonic in $\Omega\setminus \Lambda(u)$.
\end{proposition}
\begin{proof}
	In this proof let us say a subset $A\subset F$ is \emph{$F$-open} if it is open with respect to the subset topology of $F$. Similarly, we use \emph{$F$-closed} and the appropriate counterpart for $\Omega$ instead of $F$.
	
	The superharmonicity of $u$ from \Cref{prop:superharmonic} implies that $u$ is lower semi-continuous on $\Omega$ after a possible redefinition on a set of measure zero, see \cite[Lemma 1.17]{FRRO22}. Thus, the set $\{u>\psi\} \coloneqq  \{x\in F \mid u(x) > \psi(x) \}$ is $F$-open, as $u-\psi$ is lower semi-continuous on $F$.
	
	Let us now show that $u$ is harmonic away from $\{u = \psi\} \coloneqq  \{x\in F\mid u(x) = \psi(x)\}$. First, harmonicity in $\{u>\psi\}$ follows from a similar argument as in \Cref{prop:superharmonic}: Since $\{u>\psi\}$ is $F$-open, $\{u\leq \psi\} = F\setminus \{u>\psi\}$ is $F$-closed and therefore $\Omega\setminus \{u\leq \psi\}$ is open as $F$ itself is $\Omega$-closed. For any $\varphi\in C_c^\infty(\Omega\setminus\{u\leq \psi\})$ we have $\dist(\supp \varphi, \{u\leq \psi\}) > 0$. Thus, there exists some open interval $I_{\varphi}\subset\mathbb R$ that contains $0$ and $u+t\varphi \in \mathcal K_{u, \psi, F}(\Omega)$ for every $t\in I_{\varphi}$. By the variational argument from \Cref{prop:superharmonic} we conclude that $(\nabla u, \nabla\varphi)_{L^2(\Omega)} = 0$. Hence, $u$ is harmonic in $\Omega\setminus\{u\leq \psi\}$. The harmonicity in $\{u < \psi\}$ follows from the observation that $\capacity_1(\{u<\psi\};\Omega) = 0$ as $u\in \mathcal K_{u, \psi, F}(\Omega)$ and therefore
	\begin{equation*}
		\mathcal K_{u, \psi, F}(\Omega) = \mathcal K_{u, \psi, F\setminus \{u<\psi\}}(\Omega).
	\end{equation*}
	This means, $u$ coincides with the solution to the $\mathcal K_{u, \psi, F\setminus \{u<\psi\}}(\Omega)$-obstacle problem, which implies the harmonicity away from $\{u=\psi\}$ by \Cref{prop:superharmonic}.

	We now show that $u$ is continuous. Since $u$ is harmonic (thus smooth) in $\Omega\setminus \{u=\psi \}$, we claim that $u$ is continuous on $\{u=\psi\} $ as well. Indeed, let $y_0 \in \{u=\psi\}$ and let us argue by contradiction. More precisely, suppose that there exists a sequence $y_k \to y_0$ such that $u(y_k) \to u(y_0)+\varepsilon_0=\psi(y_0)+\varepsilon_0$ for some $\varepsilon_0>0$, where we used the lower semi-continuity of $u$. Since $\psi$ is upper semi-continuous, we may assume that $y_k \in \{u>\psi\}$ for $k \in \mathbb{N}$. Otherwise, we can extract a subsequence satisfying this property. Let us denote by $z_k$ the projection of $y_k$ towards $\{u=\psi\} \cap F$, that is,  $|y_k-z_k| = \inf_{z\in \{u=\psi\}} |y_k-z_k|$. We note that $z_k \in F$ for large $k$, since $F$ is $\Omega$-closed. If we set $\delta_k \coloneqq  |y_k-z_k| > 0$, then we clearly have $\delta_k \to 0$, $z_k \to y_0$, and $B_{\delta_k}(y_k) \subset B_{2\delta_k}(z_k)$. Thus, the superharmonicity of $u$ from \Cref{prop:superharmonic} implies
	\begin{equation*}
		u(z_k) \geq \dashint_{B_{2\delta_k}(z_k)} u \, dx = (1-2^{-n}) \dashint_{B_{2\delta_k}(z_k)\setminus B_{\delta_k}(y_k)} u \, dx + 2^{-n}\dashint_{B_{\delta_k}(y_k)} u \, dx \eqqcolon  I_1 + I_2.
	\end{equation*}
	In $I_2$ we use the harmonicity of $u$ away from $\{u=\psi\}$ to get $I_2 = 2^{-n}u(y_k)$. From the lower semi-continuity there exists an open neighbourhood $U$ of $y_0$ in $\Omega$ such that 
	\begin{equation}\label{eq:continuity proof 1}
		u(x) \geq u(y_0) - 2^{-n}\varepsilon_0 \quad\text{for every $x\in U$.}
	\end{equation}
	For large $k$, we have $B_{2\delta_k}(z_k) \subset U$. Thus, we can use \eqref{eq:continuity proof 1} in $I_1$ to get
	\begin{equation}\label{eq:continuity proof 2}
		u(z_k) \geq (1-2^{-n})u(y_0) - (1-2^{-n})2^{-n}\varepsilon_0 + 2^{-n}u(y_k).
	\end{equation}
	By assumption we have $\lim_{k\to\infty} u(y_k) = u(y_0) + \varepsilon_0$. If $z_k\in \{u=\psi\}$ we could use $u(z_k) = \psi(z_k)$ together with the upper semi-continuity of $\psi$ to deduce
	\begin{equation}\label{eq:continuity proof 3}
		\psi(y_0) \geq \limsup_{k\to\infty} \psi(z_k) = \limsup_{k\to\infty} u(z_k) \geq u(y_0) + 2^{-2n}\varepsilon_0 = \psi(y_0) + 2^{-2n}\varepsilon_0,
	\end{equation}
	which leads to a contradiction. However, it is unclear if  $z_k\in \{u=\psi\}$ for every $k$ (or some subsequence).  But, we know that $z_k$ is contained in the $F$-closure of $\{u=\psi\}$ for large enough $k$. In this case, we can use the lower semi-continuity of $u$ and upper semi-continuity of $\psi$ to deduce the following: For $x$ in the $F$-closure of $\{u=\psi\}$ and any sequence $x_n\in\{u=\psi\}$ converging to $x$ we have
	\begin{equation*}
		\psi(x) \geq \limsup_{n\to\infty} \psi(x_n) = \limsup_{n\to\infty} u(x_n) \geq \liminf_{n\to\infty} u(x_n) \geq u(x).
	\end{equation*}
	Hence, we have $\psi(z_k) \geq u(z_k)$ that we can use in \eqref{eq:continuity proof 2} to deduce \eqref{eq:continuity proof 3}.
\end{proof}
As a direct consequence of \Cref{thm:continuity and harmonicity}, we have shown that any solution $u\in H^1(\Omega)$ to the $\mathcal K_{u, \psi, F}(\Omega)$-obstacle problem satisfies the Euler-Lagrange equations
\begin{equation}\label{eq:euler lagrange}
	\left\{\begin{aligned}
		-\Delta u &\geq 0 &&\text{in $\Omega$} \\
		-\Delta u &= 0 &&\text{in $\Omega\setminus \Lambda(u)$} \\
		u - \psi &\geq 0 &&\text{on $F$.}
	\end{aligned}\right.
\end{equation}
In fact, these problems are equivalent.
\begin{lemma}\label{lem:euler lagrange}
	Let $\Omega\subset\mathbb R^n$ be open, $F\subset \Omega$ be relatively closed, and $\psi:F\to\overline{\mathbb R}$ be upper semi-continuous. Assume that $u\in H^1(\Omega)\cap C(\Omega)$. Then, $u$ satisfies \eqref{eq:euler lagrange} if and only if $u$ solves the $\mathcal K_{u, \psi, F}(\Omega)$-obstacle problem.
\end{lemma}
\begin{proof}
	If $u$ solves the $\mathcal K_{u, \psi, F}(\Omega)$-obstacle problem, then it satisfies \eqref{eq:euler lagrange} by \Cref{thm:continuity and harmonicity}. 
	
	To show the contrary, we first show that solutions to \eqref{eq:euler lagrange} are unique in $H^1(\Omega)\cap C(\Omega)$ respective to their boundary values. Assume we have two functions $u_1, u_2 \in H^1(\Omega)\cap C(\Omega)$ satisfying \eqref{eq:euler lagrange} and $u_1 - u_2 \in H^1_0(\Omega)$. Let $G\coloneqq  \{u_2 > u_1\}$. Note, $G$ is open due to the continuity of $u_1$ and $u_2$. Moreover, since
	\begin{equation*}
		G \subset \Omega \setminus \Lambda (u_2),
	\end{equation*}
	we observe that
	\begin{equation*}
		-\Delta u_2=0 \quad \text{in $G$.}
	\end{equation*}
	Thus, if $G\neq \emptyset$ then $w\coloneqq u_2-u_1 \in H^1_0(G)$ satisfies in a weak sense:
	\begin{equation*}
		\left\{\begin{aligned}
			-\Delta w &\leq 0 &&\text{in $G$}\\
			w &= 0&&\text{on $\partial G$.}
		\end{aligned}\right.
	\end{equation*}
	Then, the maximum principle yields that $w=u_2-u_1 \leq 0$ in $G$, which leads to a contradiction. As the same argument holds if we exchange the roles of $u_1$ and $u_2$ in the definition of $G$, we conclude $u_1 = u_2$.
	
	To finish the proof, let $u\in H^1(\Omega)\cap C(\Omega)$ satisfy \eqref{eq:euler lagrange}. Then, $u\in \mathcal K_{u, \psi, F}(\Omega)$ and \Cref{thm:well posedness} gives the existence of a solution $v\in H^1(\Omega)$ to the $\mathcal K_{u, \psi, F}(\Omega)$-obstacle problem. By \Cref{thm:continuity and harmonicity}, $v$ satisfies \eqref{eq:euler lagrange} and we have $u-v \in H^1_0(\Omega)$. Thus, from the previous line of arguments, $u=v$ and therefore $u$ solves the $\mathcal K_{u, \psi, F}(\Omega)$-obstacle problem.
\end{proof}

\subsection{H\"{o}lder and Lipschitz regularity}
The aim of this section is to discuss preliminary results on the regularity of solutions to the obstacle problem implied by known regularity results for associated Dirichlet boundary value problems. To be precise, we show H\"{o}lder regularity of some order in the case that $F$ satisfies a capacity density condition and Lipschitz regularity in the case that $F$ satisfies a uniform exterior ball condition.

\begin{definition}\label{def:cdc and outer ball} \
	\begin{enumerate}
		\item We say that a set $E\subset \mathbb R^n$ satisfies a \emph{capacity density condition} if there exist constants $c_0>0$ and $r_0>0$ such that
		\begin{equation}\label{eq-capdensity}
			\frac{\mathrm{Cap}_0(E \cap B_r(x_0), B_{2r}(x_0))}{\mathrm{Cap}_0(B_r(x_0), B_{2r}(x_0))}   \geq c_0
		\end{equation}
		whenever $0<r<r_0$ and $x_0 \in E$.
		\item We say that a set $E\subset \mathbb R^n$ satisfies a \emph{uniform exterior ball condition} (with radius $r_0$) if there exists $r_0>0$ satisfying the following condition: For every $x_0 \in \partial E$, there exists a point $y_0 \in E^c$ such that $B_{r_0}(y_0) \subset E^c$ and $x_0 \in \partial B_{r_0}(y_0)$.
	\end{enumerate}
\end{definition}
While the exterior ball condition is quite intuitive, it is noteworthy that there are also several geometric conditions that imply \eqref{eq-capdensity}. We collect a small list in \Cref{ex:cdc and outer ball}. Now, let us prove H\"{o}lder and Lipschitz regularity of solutions to obstacle problems when a capacity density condition or an exterior ball condition holds, respectively.
\begin{theorem}\label{thm:preliminary holder reg}
	Let $\Omega\subset \mathbb R^n$ be open and bounded, $F\subset \Omega$ be relatively closed, and $\psi : F \to \overline{\mathbb R}$. Let $u\in H^1(\Omega)$ be a bounded solution to the $\mathcal K_{u, \psi, F}(\Omega)$-obstacle problem. Assume there exists $\Psi \in H^1(\Omega\setminus F)$ such that
	\begin{equation*}
		\Psi = \psi \text{ q.e.\ on $F$}\quad\text{and}\quad \Psi \leq u \text{ on $\partial\Omega$.}
	\end{equation*}
	\begin{enumerate}
		\item If $\Psi\in C^{\alpha}(\partial(\Omega\setminus F))$ for some $\alpha \in (0, 1)$ and $(\Omega\setminus F)^c$ satisfies the capacity density condition \eqref{eq-capdensity}, then there exists $\alpha_0 \in (0, \alpha]$ such that $u$ is $C^{\alpha_0}$-regular in $\Omega$. Moreover, for any compact set $K\subset\Omega$ there is  $C>0$ such that
		\begin{equation*}
			\norm{u}_{C^{\alpha_0}(K)} \leq C(\norm{u}_{L^\infty(\Omega)}+\norm{\Psi}_{C^{\alpha}(\partial(\Omega \setminus F))}).
		\end{equation*}
		\item If $\Psi \in \lip(\partial(\Omega\setminus F))$ and $\Omega\setminus F$ satisfies a uniform exterior ball condition, then $u$ is Lipschitz continuous in $\Omega$. Moreover, for any compact set $K\subset\Omega$ there is  $C>0$ such that
		\begin{equation*}
			\norm{u}_{\lip(K)} \leq C(\norm{u}_{L^\infty(\Omega)}+\norm{\Psi}_{\lip(\partial(\Omega \setminus F))}).
		\end{equation*}
	\end{enumerate}
\end{theorem}

\begin{proof}
	First, assume that $\Psi$ is H\"{o}lder continuous and $(\Omega\setminus F)^c$ satisfies \eqref{eq-capdensity}. Let $v\in H^1(\Omega\setminus F)$ be the unique weak solution to the Dirichlet boundary value problem
	\begin{equation*}
		\left\{\begin{aligned}
			-\Delta v &=  0 &&\text{in $\Omega\setminus F$}\\
			v &= \Psi &&\text{on $\partial(\Omega\setminus F)$.}
		\end{aligned}\right.
	\end{equation*}
	Since $\Psi \in C^\alpha(\partial(\Omega\setminus F))$, \cite[Theorem 6.44]{HKM06} implies the existence of $\alpha_0 \in (0, \alpha]$ such that $v\in C^{\alpha_0}(\overline{\Omega\setminus F})$. Let
	\begin{equation*}
		\overline v \coloneqq  \begin{cases}
			v &\text{in $\Omega\setminus F$} \\
			\psi &\text{on $F$.}
		\end{cases}
	\end{equation*}
	As $u \geq \Psi = v$ on $\partial(\Omega\setminus F)$ the comparison principle implies that $u \geq \overline v = v$ in $\Omega\setminus F$. Clearly, $u \geq \overline v = \psi$ on $F$. Thus, $u$ solves the $\mathcal K_{u, \overline v, \Omega}(\Omega)$-obstacle problem. Since it corresponds to a classical obstacle problem, \cite[Theorem 2]{Caf98} implies that $u$ is as regular as the obstacle $\overline v$ as long as we stay away from $\partial\Omega$. Thus, we conclude that $u\in C^{\alpha_0}(K)$ for any compact subset $K\subset \Omega$ with the claimed bound.
	
	In the case that $\Psi \in \lip(\partial(\Omega \setminus F))$ and $\Omega\setminus F$ satisfies a uniform exterior ball condition the line of arguments stays the same but we replace \cite[Theorem 6.44]{HKM06} by a variant of \cite[Lemma 1.2]{Saf08}, see \Cref{lem:safonov lemma} below.
\end{proof}
\begin{lemma}\label{lem:safonov lemma}
	Let $\Omega\subset\mathbb R^n$ be open and bounded. Assume that $\Omega$ satisfies a uniform exterior ball condition. Then, every $u\in H^1(\Omega)\cap C(\overline{\Omega})$ such that $-\Delta u = 0$ and $u\vert_{\partial\Omega} \in \lip(\partial\Omega)$ is Lipschitz on $\overline{\Omega}$ and satisfies the estimate
	\begin{equation*}
		\norm{u}_{\lip(\overline{\Omega})} \leq C\left(\norm{u}_{L^\infty(\Omega)} + \norm{u}_{\lip(\partial\Omega)} \right)
	\end{equation*}
	for a constant $C>0$ depending on $n$, $\Omega$, and $r_0$ from the uniform exterior ball condition.
\end{lemma}
\begin{proof}[Sketch of proof]
	Use a barrier approach as in \cite[Lemma 1.2]{Saf08}: For $x_0\in \partial\Omega$ consider the functions
	\begin{equation*}
		\varphi_{\pm}(x) \coloneqq  u(x_0) \pm c_{\pm}\left(r_0^{-\lambda} - |x-y_0|^{-\lambda} \right),
	\end{equation*}
	where $r_0$ is the radius of the uniform exterior ball condition, $y_0$ is the center of the exterior ball at $x_0$, $\lambda \geq n-2$, and $c_{\pm}$ are positive constants chosen large enough such that $\varphi_{-} \leq u \leq \varphi_{+}$ on $\partial\Omega$. Then, the claimed result follows from the comparison principle.
\end{proof}

\begin{remark} \
	\begin{enumerate}
		\item The boundedness of solutions to the obstacle problem follows immediately from the boundedness of their boundary data as well as the boundedness of the obstacle (from above) due to the comparison principle, see for example \cite[Theorem 3.24]{HKM06} or \cite[Proposition 7.5]{BB11}. On the other hand, due to \Cref{thm:continuity and harmonicity} we already know that the solution $u$ of an obstacle problem in $\Omega$ with upper semi-continuous obstacle $\psi$ is continuous in $\Omega$, which implies that it is bounded in any compact subset.
		\item In \Cref{thm:preliminary holder reg} $(i)$, a sufficient condition such that $(\Omega\setminus F)^c$ satisfies a capacity density condition is that both $\Omega^c$ and $F$ satisfy a capacity density condition. In \Cref{thm:preliminary holder reg} $(ii)$, $\Omega\setminus F$ satisfies a uniform exterior ball condition if $\Omega$ satisfies a uniform exterior ball condition and $F$ satisfies a uniform interior ball condition.
        \item As the proof of \Cref{thm:preliminary holder reg} depends only on the regularity of an associated Dirichlet boundary value problem in $\Omega\setminus F$, it is expected that the Lipschitz regularity holds in the case that $\Omega\setminus F$ satisfies a uniform exterior $C^{1, \mathrm{Dini}}$-condition, see \cite{KK73, KK74}. In particular, the exterior $C^{1, \mathrm{Dini}}$-condition is necessary for the Lipschitz regularity, see \cite[Theorem 2]{KK74}. See also \cite{HLW14} and the references therein.
	\end{enumerate}
	
\end{remark}

\begin{examples}\label{ex:cdc and outer ball}
	Here, let us collect some examples for sets $F$ that satisfy a uniform exterior ball or capacity density condition.
	\begin{enumerate}
		\item It is well known, see for example \cite[Lemma 2.2]{AKSZ07}, that a bounded domain is a $C^{1, 1}$-domain if and only if it satisfies a uniform (interior and exterior) ball condition.
		\item A uniform exterior ball condition is satisfied by many domains having bad interior regularity, such as any convex Lipschitz domain or a convex domain with an outward cusp.
		\item If $F$ contains a cone (or a corkscrew) in a uniform sense at each point $x_0 \in \partial F$, then $F$ satisfies a capacity density condition, see \cite[Theorem 6.31]{HKM06} for instance. In particular, any Lipschitz domain satisfies a capacity density condition.
		\item A large class of fractals satisfy a capacity density condition. More precisely, any nonempty self-similar compact set given as the attractor of an iterated function system of similitudes (see for example \cite{Fal14}) with Hausdorff-dimension $d\in (n-2, n]$ satisfies a capacity density condition. This includes, ignoring the dimensional constraint, for example Cantor sets, Sierpi\'{n}ski gaskets and carpets, the Vicsek set, and so on. A proof of this fact may be found in \cite[Section 3]{ACM22} in the particular case of the middle third Cantor set. However, the proof in the general case is the same.
	\end{enumerate}
\end{examples}
This list of examples shows, focusing on the obstacle, that whenever $F$ and $\psi$ are \enquote{nice} in the sense that $F$ has a nonempty interior and satisfies a uniform interior ball condition and $\psi$ is Lipschitz continuous, we can expect the solutions to the associated obstacle problems to be of at least Lipschitz regularity.

On the other hand, as the capacity density condition is satisfied by a large variety of sets $F$ we can expect some H\"{o}lder regularity in most cases of interest. In the upcoming sections we discuss \emph{Signorini-type problems}, that is, $F$ is contained in an $(n-1)$-dimensional manifold that separates $\Omega$ into two disjoint parts. Our results, in particular \Cref{thm:optimalreg}, suggest that the optimal (local) regularity of solutions to such Signorini-type problems is $C^{1\slash 2}$ under an additional capacity density condition on $F$.

One might expect better regularity if $F$ is not contained in an $(n-1)$-dimensional manifold. However, here it seems that a classification of obstacle sets $F$ and their associated optimal regularity is more delicate as it involves the angles of $\partial F$, similar to the boundary regularity problem for Dirichlet boundary value problems. To aid that claim, let us extend the standard examples for thin-obstacle problems to specific Lipschitz domains $F$.

Fix $n=2$ and consider for every $0< \alpha \leq 1$ the function $h_{\alpha}(x, y) = -\re(x+i|y|)^{\alpha}$ or represented in polar coordinates by
\begin{equation*}
	h_\alpha(r, \theta) = -r^\alpha \cos(\alpha\theta).
\end{equation*}
It is easy to see, using the Laplacian in polar coordinates, that $h_\alpha$ is harmonic in the upper and lower half-plane. Furthermore, the normal derivative of $h_{\alpha}$ on the right side $\{x>0, y=0 \}$ is zero. Thus, $h_{\alpha}$ is harmonic along the right side, which means
\begin{equation*}
	-\Delta h_{\alpha} = 0 \quad\text{in $\mathbb R^2\setminus \{x\leq 0, y=0 \}$.}
\end{equation*}
Moreover, we can see for $r>0$ and $\theta \in [0, \pi]$ that
\begin{equation*}
	h_{\alpha}(r, \theta) \geq 0 \quad \text{if and only if}\quad \theta \in \left[\frac{\pi}{2\alpha}, \pi \right]
\end{equation*}
and
\begin{equation*}
	h_{\alpha}(r, \theta) = 0 \quad\text{if and only if}\quad \theta = \frac{\pi}{2\alpha}.
\end{equation*}
Thus, $h_{\alpha}$ can be expected to solve a zero obstacle problem around the origin only if $\alpha \geq 1\slash 2$ as else we have $h_{\alpha} < 0$ on $\{x< 0, y=0 \}$.
Hence, we can easily determine when $h_{\alpha}$, $\alpha \geq 1\slash 2$, solves some zero obstacle problem in $\Omega = B_1(0)$: Let $F_{\alpha}$ be the closed cone at the origin that opens to the left with aperture $\phi=\pi(1-\frac{1}{2\alpha})$, see \Cref{figure-triangle}. Then, $h_{\alpha}$ solves the $\mathcal K_{h_{\alpha}, 0, F_{\alpha}}(B_1(0))$-obstacle problem. 
\begin{figure}[H]
	\centering
	\begin{tikzpicture}
		\node[anchor=south west, inner sep=0] at (0, 0)  {\includegraphics[width=.25\textwidth]{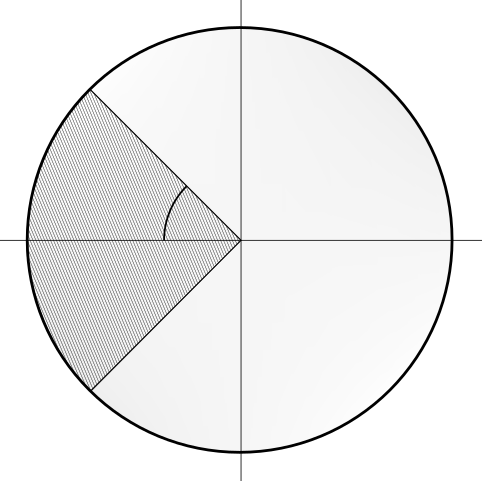}};
		\node[left] at (1.9, 2.2) {$\phi$};
		\node[above left] at (1, 2.1) {$F_{\alpha}$};
	\end{tikzpicture}
	\caption{The unit ball with cone $F_\alpha$.}
	\label{figure-triangle}
\end{figure}
Clearly, $h_{\alpha}$ is $C^{\alpha}$-regular at the origin. In particular, when $\alpha = 1\slash 2$, the cone $F_{1\slash 2}$ becomes simply the left half-line $\{x\leq 0, y=0 \}$. This simple example limits the optimal regularity in the Signorini-type case to $C^{1\slash 2}$.

For any $\alpha \in (1\slash 2, 1)$, $F_{\alpha}$ is a Lipschitz domain and the regularity at the origin clearly depends on the aperture of the cone. That means, similar to the boundary regularity problem for Dirichlet boundary value problems on Lipschitz domains, the optimal regularity for the obstacle problem in the case that $F$ is a Lipschitz domain depends heavily on the Lipschitz constants of the representing chart as this determines the possible range of angles. In particular, if the aperture of the cone is larger, that is $\phi \geq \pi \slash 2$, then we expect higher regularity than Lipschitz as it is known for the Dirichlet problem. For this, we refer to \cite{MR10} for the weighted $L^p$-estimates and related embedding results concerning Dirichlet boundary value problems on Lipschitz domains, see also \cite{JK95}.

\section{Optimal H\"{o}lder regularity for Signorini-type problems under a capacity density condition}\label{sec-optimalholder}

In this section we consider the unit ball $B_1 \coloneqq  B_1(0) \subset \mathbb R^n$, $n\geq 2$, and a relatively closed subset $F \subset B_1$ supported on the horizontal hyperplane $\{x_n = 0\}$, that is, $F\subset B_1\cap \{x\in B_1 \mid x_n = 0\}$. Let us abbreviate $B_r' \coloneqq  B_r(0)\cap \{x\in B_r(0) \mid x_n = 0\}$. With these special choices of $\Omega$ and $F$, we call the corresponding $\mathcal K_{g, \psi, F}(B_1)$-obstacle problem, for boundary data $g$ and obstacle $\psi$, a \emph{Signorini-type problem}.

We now investigate the local regularity of a solution $u$ to the Signorini-type problem for the zero obstacle $\psi = 0$. We consider non-zero obstacles at the end of the section in \Cref{thm:optimalreg non zero}. Let $u\in H^1(B_1)$ be a solution to the $\mathcal K_{u, 0, F}(B_1)$-obstacle problem. If $F$ coincides with the full horizontal hyperplane, that is, $F=B_1\cap \{x\in B_1 \mid x_n = 0\}$, then $u$ is a solution to the classical Signorini problem. In this case, as already mentioned, the optimal regularity of $u$ is well known, that is, $u$ is Lipschitz continuous along the hyperplane and $C^{1, 1\slash 2}$ from above and below up to the hyperplane. However, in the case that $F$ is a proper subset, the example at the end of \Cref{sec-generalobstacle} shows that we can not expect better than $C^{1\slash 2}$-regularity.

Recall, the contact set is denoted by $\Lambda(u) = \{x\in F \mid u(x) = 0 \}$. The main result of this section is:
\begin{theorem}\label{thm:optimalreg}
	Let $n\geq 2$ and $F\subset \{x\in B_1 \mid x_n = 0\}$ be relatively closed in $B_1$. Assume that $u\in H^1(B_1)$ is a solution to the $\mathcal K_{u, 0, F}(B_1)$-obstacle problem. If $\Lambda(u)$ satisfies a capacity density condition at $x_0\in \Lambda(u)$, that is, there exists a constant $c_0 > 0$ and a radius $r_0>0$ such that
	\begin{equation}\label{eq:CDC contact set}
		\capacity_0(\Lambda(u) \cap B_r(x_0); B_{2r}(x_0)) \geq r^{n-2}c_0 \quad\text{for every $r\in(0, r_0)$,}
	\end{equation}
	then $u$ is $C^{1\slash 2}$-regular at $x_0$.
\end{theorem}
The remainder of this section is devoted to the proof of \Cref{thm:optimalreg}, but let us first make some comments on \Cref{thm:optimalreg} and its proof. As the contact set $\Lambda(u)$ is unknown a priori, \eqref{eq:CDC contact set} is difficult to verify. However, we want to emphasize the fact that \eqref{eq:CDC contact set} is not as restrictive as it may seem. As already mentioned in \Cref{ex:cdc and outer ball} $(iii)$ and $(iv)$ a uniform capacity density condition, in the sense of \Cref{def:cdc and outer ball}, is satisfied by a large variety of sets including various fractals. Note, \eqref{eq:CDC contact set} is truly a local variant of \Cref{def:cdc and outer ball} as $\capacity_0(B_r(x_0), B_{2r}(x_0)) = cr^{n-2}$ for some constant $c$ independent of $r$, see for example \cite[(2.13)]{HKM06}. Thus, it can be expected that \eqref{eq:CDC contact set} holds for a large class of solutions.

Nonetheless, let us discuss some of the contrary cases. First, in a simple case when \eqref{eq:CDC contact set} is not satisfied is the following. Assume $u\in H^1(B_1)$ solves the $\mathcal K_{u, 0, F}(B_1)$-obstacle problem such that there exists $\delta>0$ satisfying $\Lambda(u)\cap B_{\delta}(0) = \{0\}$. Clearly, \eqref{eq:CDC contact set} fails in $x_0 = 0$. However, $u$ is harmonic and bounded in $B_{\delta}(0)\setminus \{0\}$ due to \Cref{thm:continuity and harmonicity}. Thus, by the removable singularity theorem for harmonic functions, $u$ extends uniquely to a harmonic function in $B_{\delta}(0)$ and is therefore smooth in $x_0=0$. A more general version of this removable singularity result, and therefore smoothness of the solutions, also holds if $\Lambda(u)\cap B_{\delta}(x_0)$ is of zero capacity, see \cite{Ser64}.

Thus, it remains to study the case when $r^{-(n-2)}\capacity_0(\Lambda(u) \cap B_r(x_0); B_{2r}(x_0))$ is a proper sequence converging to zero. We comment on this case in \Cref{rem:capacity converging to 0}.
\begin{remark}
    A variant of \eqref{eq:CDC contact set} also appears in the articles \cite{ACM22, AC22, AC23} within the context of mixed boundary value problems. See \Cref{sec-almostoptimal} and \Cref{rem:AC remark} for the relation of these articles to Signorini-type problems.
\end{remark}

The strategy of the proof of \Cref{thm:optimalreg} is similar to the proof of \cite[Theorem 5]{AC04}. First, in \Cref{subsec:acf} we introduce an Alt--Caffarelli--Friedman-type frequency function and show its monotonicity in \Cref{lem:acf frequency}. The main difficulty here is the lack of regularity along the hyperplane. The key observation is that the frequency function is absolutely continuous. Thus, any equation involving the derivative of the frequency function holds  in an almost everywhere sense. Then, we conclude the proof in \Cref{subsec:proof of optimal reg} with the help of Maz'ya's inequality, see \Cref{lem:mazya}, and the subharmonicity of the positive and negative part of the solution.

\begin{remark}\label{rem:acf function} \
    \begin{enumerate}
        \item The Alt--Caffarelli--Friedman frequency function was first introduced in \cite{ACF84} in the context of a two-phase free boundary problem. There, it is given by
        \begin{equation*}
            J(r, u) = \frac{1}{r^4} \int_{B_r} \frac{|\nabla u_+|^2}{|x|^{n-2}} \, dx \int_{B_r} \frac{|\nabla u_-|^2}{|x|^{n-2}} \, dx,
        \end{equation*}
        where $u$ denotes the solution of a certain free boundary problem. The monotonicity of $J$ allows the authors to prove regularity of solutions and identify blowup limits.
        \item In \cite{AC04}, see also the lecture notes \cite{Sal12}, the authors introduces a variant of the Alt--Caffarelli--Friedman frequency, namely
        \begin{equation*}
            \beta(r, w) = \frac{1}{r} \int_{B_r^+} \frac{|\nabla w|^2}{|x|^{n-2}} \, dx,
        \end{equation*}
        to study thin obstacle problems. One major difference to the Alt--Caffarelli--Friedman frequency is that $w$ represents $\partial_{x_n} u$, where $u$ solves a global thin obstacle problem. Additionally, it is assumed that $u$ is convex along every direction tangential to the hyperplane $\{x_n = 0\}$. The monotonicity of $\beta$ then implies the $C^{1\slash 2}$-regularity of $w$, which implies the optimal  $C^{1, 1\slash 2}$-regularity of the solution $u$ in $B_1^+$ up to the hyperplane.
    \end{enumerate}
\end{remark}

\subsection{Alt--Caffarelli--Friedman-type frequency function}\label{subsec:acf}
\begin{proposition}[Monotonicity of Alt--Caffarelli--Friedman-type frequency function]\label{lem:acf frequency}
	Let $F\subset B_1\cap \{x_n = 0\}$ be relatively closed in $B_1$ and assume that $u\in H^1(B_1)$ is a solution to the $\mathcal K_{u, 0, F}(B_1)$-obstacle problem. The function
	\begin{equation*}
		\beta(r) = \beta(r, u) \coloneqq  \frac{1}{r} \int_{B_r} \frac{|\nabla u|^2}{|x|^{n-2}} \, dx \quad\text{for $r\in (0, 1)$},
	\end{equation*}
	is nonnegative and finite for every $r\in (0, 1)$ and $r\mapsto \beta(r)$ is continuous and nondecreasing.
\end{proposition}
The proof of \Cref{lem:acf frequency} is similar to \cite[Lemma 4]{AC04}. That means, we take the derivative of $\beta$ with respect to $r$ and prove its nonnegativity using a lower bound on the Rayleigh quotient associated with the spherical Laplacian. As already mentioned, the main difficulty is the lack of regularity along the hyperplane $\{x_n = 0\}$. The key observation is that $\beta$ is absolutely continuous. Thus, every identity that involves the derivative of $\beta$ holds in an almost everywhere sense. 

An important ingredient in the proof of \Cref{lem:acf frequency} is the following integration by parts formula:
\begin{lemma}\label{lem:integration by parts u}
    Let $F\subset B_1\cap \{x_n = 0\}$ be relatively closed in $B_1$ and assume that $u\in H^1(B_1)$ is a solution to the $\mathcal K_{u, 0, F}(B_1)$-obstacle problem. Then, for almost every $0<r<1$ we have
	\begin{equation}\label{eq:integration by parts u}
		\int_{B_r} \nabla u \cdot \nabla \varphi \, dx = \int_{\partial B_r} \varphi \partial_{\nu} u \, d\sigma \quad\text{for every $\varphi\in H^1(B_1)$ such that $\varphi = 0$ q.e.\ on $\Lambda(u)$.}
	\end{equation}
    Here, $\sigma$ denotes the surface measure.
\end{lemma}
Before we prove \Cref{lem:integration by parts u} and \Cref{lem:acf frequency} let us make some simple observations for solutions $u$ to such Signorini-type problems:
\begin{enumerate}
	\item The gradient $\nabla u(x)$ is well defined pointwise for every $x\in B_1\setminus \Lambda(u)$.
	\item For every $r\in(0, 1)$ the normal derivative $\partial_{\nu} u(x) = \nabla u (x) \cdot \vec \nu(x)$ is well defined for $x\in \partial B_r\setminus B_1'$ where $\vec \nu(x)$ denotes the outward normal with respect to $B_r$.
	\item The maps $r\mapsto \int_{\partial B_r} |\nabla u|^2 \, d\sigma$ and $r\mapsto \int_{\partial B_r} |\partial_{\nu} u|^2 \, d\sigma$ are in $L^1(0, 1)$.
	\item For almost every $r\in (0, 1)$ we have $|\nabla u|, \partial_{\nu} u \in L^2(\partial B_r, \sigma)$.
\end{enumerate}
We note that items $(i)$ and $(ii)$ follow from the fact that $u$ is harmonic (and therefore smooth) away from the contact set $\Lambda(u)$. In particular, $\nabla u$ and $\partial_{\nu} u$ are well defined pointwise in $B_1\setminus B_1'$. Item $(iii)$ follows by using $|\partial_{\nu} u|^2 \leq |\nabla u|^2$ and $u\in H^1(B_1)$ and $(iv)$ is implied by $(iii)$.
\begin{proof}[Proof of \Cref{lem:integration by parts u}]
	Let us assume that $u$ is symmetric with respect to the hyperplane $\{x_n = 0\}$. Since $u$ is harmonic in $B_1\setminus \Lambda(u)$ it is smooth up to the boundary in
	\begin{equation*}
		B_{r, \varepsilon}^+ \coloneqq  \{x = (x', x_n) \in B_r \mid x_n >\varepsilon \} \quad\text{for $\varepsilon>0$.}
	\end{equation*}
	For any $\varphi\in C_c^\infty(B_1\setminus \Lambda(u))$, integration by parts implies
	\begin{equation}\label{eq:integration by parts u proof 0}
		\int_{B_r^+} \nabla u \cdot\nabla \varphi \, dx = \lim_{\varepsilon\to 0} \int_{B_{r, \varepsilon}^+} \nabla u \cdot\nabla \varphi \, dx  = \lim_{\varepsilon\to 0} \int_{\partial B_{r, \varepsilon}^+} \varphi \partial_{\nu} u \, d\sigma. 
	\end{equation}
	We then separate $\partial B_{r, \varepsilon}^+$ into
	\begin{equation*}
		(\partial B_r)_{\varepsilon}^+ = \{ x\in \partial B_r\mid x\in\partial B_{r, \varepsilon}^+  \} \quad\text{and}\quad (B_r')_{\varepsilon}^{+} = \{ x = (x', \varepsilon) \mid |x'| < r \}.
	\end{equation*}
	Let us now show that
	\begin{equation}\label{eq:integration by parts u proof 1}
		\lim_{\varepsilon\to 0} \int_{(B_r')_{\varepsilon}^{+} } \varphi \partial_{\nu} u \, d\sigma = 0.
	\end{equation}
	There exists a smooth open neighbourhood $U$ of $\Lambda(u)$ such that $\varphi = 0$ on $U$, and thus 
	\begin{equation*}
		\int_{(B_r')_{\varepsilon}^{+}\cap U} \varphi \partial_{\nu} u \, d\sigma  = 0.
	\end{equation*}
	Since $u$ is harmonic in $B_1\setminus U$, we have
	\begin{equation*}
		\lim_{\varepsilon\to 0} \int_{(B_r')_{\varepsilon}^{+} \setminus U} \varphi \partial_{\nu} u \, d\sigma = \int_{B_r' \setminus U} \varphi \partial_{\nu} u \, d\sigma.
	\end{equation*}
	 The symmetry of $u$ together with the harmonicity implies $\partial_{\nu} u = 0$ on $B_r' \setminus U$. Thus, we arrive at \eqref{eq:integration by parts u proof 1}. 
     
     Similarly, for $(\partial B_r)_{\varepsilon}^+$ we have
	\begin{equation}\label{eq:integration by parts u proof 2}
		\lim_{\varepsilon\to 0} \int_{(\partial B_r)_{\varepsilon}^+} \varphi \partial_{\nu} u \, d\sigma = \lim_{\varepsilon\to 0} \int_{(\partial B_r)_{\varepsilon}^+ \cap U } \varphi \partial_{\nu} u \, d\sigma + \int_{(\partial B_r)_{\varepsilon}^+\setminus U} \varphi \partial_{\nu} u \, d\sigma = 0 + \int_{(\partial B_r)^+} \varphi \partial_{\nu} u \, d\sigma.
	\end{equation}
	Thus, using \eqref{eq:integration by parts u proof 1} and \eqref{eq:integration by parts u proof 2} in \eqref{eq:integration by parts u proof 0} imply
	\begin{equation}\label{eq:integration by parts u proof 4}
		\int_{B_r^+} \nabla u \cdot \nabla \varphi \, dx  = \int_{(\partial B_r)^+} \varphi \partial_{\nu} u \, d\sigma \quad\text{for every $\varphi\in C_c^\infty(B_1\setminus \Lambda(u))$}.
	\end{equation}
	By the same arguments, \eqref{eq:integration by parts u proof 4} is also true when the upper half ball $B_r^+$ is replaced by the lower half $B_r^-$. Now, for any $r\in (0, 1)$ such that $\partial_{\nu} u \in L^2(\partial B_r)$ we conclude \eqref{eq:integration by parts u} for every $\varphi\in H^1(B_1)$ such that $\varphi = 0$ q.e.\ on $\Lambda(u)$ by approximation with functions in $C_c^\infty(B_1\setminus \Lambda(u))$.
	
	In the case when $u$ is not symmetric the proof is analogous. Let us only comment on the crucial difference. The normal derivative of $u$ at $B_r'\setminus \Lambda(u)$ does not need to vanish. However the sum of normal derivatives $\partial_{\nu}^+ u$ and $\partial_{\nu}^- u$ calculated from $B_r^+$ and $B_r^-$, respectively, vanishes. Thus, \eqref{eq:integration by parts u proof 1} and  \eqref{eq:integration by parts u proof 4} stay true when considering the sum of integrals with respect to $B_r^+$ and $B_r^-$.
\end{proof}

The proof of the monotonicity of the Alt--Caffarelli--Friedman-type frequency $\beta$ in \Cref{lem:acf frequency} consists of the following three steps:
\begin{enumerate}
	\item We show that $\beta \in [0, +\infty)$.
	\item We show that $\beta$ is absolutely continuous on $(0, 1)$.
	\item We show that $\beta' \geq 0$ almost everywhere.
\end{enumerate}

\begin{proof}[Proof of \Cref{lem:acf frequency}]
	First, let us show that $\beta(r) < + \infty$ for every $0<r< 1$. For this purpose, choose an approximation $g_\varepsilon \in C^\infty(\mathbb R^n)$ for $g(x) = |x|^{-(n-2)}$ and $\varepsilon>0$ such that:
	\begin{enumerate}
		\item $g_\varepsilon = g$ on $B_\varepsilon^c$.
		\item $g_\varepsilon \nearrow g$ pointwise.
		\item $-\Delta g_\varepsilon  \to c\delta_0$, where $\delta_0$ is the Dirac delta measure at $x = 0$ and $c>0$ is some constant depending only on the dimension.
	\end{enumerate}
	By the chain rule and integration by parts we have
	\begin{equation*}
		\int_{B_r} |\nabla u|^2 \varphi \, dx = \int_{B_r} \nabla u \cdot \nabla (u\varphi) \, dx + \frac 12 \int_{B_r} u^2 \Delta\varphi \, dx - \frac 12 \int_{\partial B_r} u^2 \partial_{\nu} \varphi \, d\sigma
	\end{equation*}
	for every $\varphi\in C^\infty(B_1)$ and $0<r<1$. In particular, for $\varphi = g_\varepsilon$ with $\varepsilon < r$ we have
	\begin{equation}\label{eq:acf monotonicity proof 00}
		\int_{B_r} |\nabla u|^2 g_\varepsilon\, dx = \int_{B_r} \nabla u \cdot \nabla (ug_\varepsilon) \, dx - \frac 12 \int_{B_r} u^2 (-\Delta g_\varepsilon) \, dx  +\frac{n-2}{2r^{n-1}} \int_{\partial B_r} u^2 \, d\sigma.
	\end{equation}
	Since $ug_\varepsilon \in H^1(B_1)$ and $ug_\varepsilon = 0$ on $\Lambda(u)$,  \Cref{lem:integration by parts u} allows us to rewrite the first term on the right-hand side to
	\begin{equation}\label{eq:acf monotonicity proof 0}
		\int_{B_r} \nabla u \cdot \nabla (ug_\varepsilon) \, dx  = \int_{\partial B_r} u g_\varepsilon \partial_{\nu} u \, d\sigma =   \frac{1}{r^{n-2}}\int_{\partial B_r} u \partial_{\nu} u \, d\sigma\quad\text{for almost every $\varepsilon < r < 1$.}
	\end{equation}
	Taking the limit $\varepsilon\searrow 0$ we conclude from \eqref{eq:acf monotonicity proof 00} and \eqref{eq:acf monotonicity proof 0} that
	\begin{equation}\label{eq:acf monotonicity proof 1}
		\int_{B_r} \frac{|\nabla u|^2}{|x|^{n-2}} \, dx = \frac{1}{r^{n-2}} \int_{\partial B_r} u \partial_{\nu} u \, d\sigma - \frac c2 u(0)^2 + \frac{ n-2}{2r^{n-1}} \int_{\partial B_r} u^2 \, d\sigma \quad\text{for almost every $0<r<1$.}
	\end{equation}
	In particular, there exists a radius $r$ arbitrarily close to 1 such that the right-hand side of \eqref{eq:acf monotonicity proof 1} exists. Since $r\mapsto \int_{B_r} \frac{|\nabla u|^2}{|x|^{n-2}} \, dx$ is nondecreasing, it is finite for every $0<r<1$ due to  \eqref{eq:acf monotonicity proof 1} and therefore $\beta(r) < + \infty$ for every $0<r<1$.
	
	Next, let us show that $\beta$ is absolutely continuous on $(0, 1)$ with almost everywhere derivative
	\begin{equation}\label{eq:acf monotonicity proof 2}
		\beta'(r) = \frac{1}{r^{n-1}} \int_{\partial B_r} |\nabla u|^2 \, d\sigma - \frac{1}{r^2} \int_{B_r} \frac{|\nabla u|^2}{|x|^{n-2}} \, dx.
	\end{equation}
	It is immediate that $r\mapsto \int_{B_r} \frac{|\nabla u |^2}{|x|^{n-2}} \, dx$ is absolutely continuous as
	\begin{equation*}
		s\mapsto \int_{\partial B_s} \frac{|\nabla u |^2}{|x|^{n-2}} \, d\sigma \in L^1\left(0, R\right) \quad\text{for every $0<R<1$}
	\end{equation*}
	and
	\begin{equation*}
		\int_{B_r} \frac{|\nabla u |^2}{|x|^{n-2}} \, dx = \int_{B_{r_0}} \frac{|\nabla u |^2}{|x|^{n-2}} \, dx + \int_{r_0}^{r} \int_{\partial B_s} \frac{|\nabla u |^2}{|x|^{n-2}} \, d\sigma ds \quad\text{for every $0<r_0<r<1$.}
	\end{equation*}
	In particular,
	\begin{equation*}
		\frac{d}{dr} \int_{B_r} \frac{|\nabla u |^2}{|x|^{n-2}} \, dx = \int_{\partial B_r} \frac{|\nabla u |^2}{|x|^{n-2}} \, d\sigma =  \frac{1}{r^{n-2}}\int_{\partial B_r} |\nabla u |^2 \, d\sigma \quad\text{for almost every $0<r<1$}.
	\end{equation*}
	Thus, the absolute continuity of $\beta$ and \eqref{eq:acf monotonicity proof 2} follow.

	At last we show $\beta' \geq 0$ almost everywhere, which finishes the proof. Fix some $0<r<1$ such that \eqref{eq:acf monotonicity proof 1} and \eqref{eq:acf monotonicity proof 2} hold. In the first term in the right-hand side of \eqref{eq:acf monotonicity proof 2} we represent $\nabla$ on $\partial B_r$ by the tangential gradient $\nabla_{\theta}u \coloneqq  \nabla u - (\partial_{\nu} u )\vec \nu$ for the outer normal vector $\vec \nu$ with respect to $B_r$, which gives:
	\begin{equation}\label{eq:acf monotonicity proof 3}
		\int_{\partial B_r} |\nabla u |^2 \, d\sigma = \int_{\partial B_r} |\nabla_{\theta} u |^2 \, d\sigma + \int_{\partial B_r} (\partial_{\nu} u)^2 \, d\sigma.
	\end{equation}
	For the second term in \eqref{eq:acf monotonicity proof 2} we use \eqref{eq:acf monotonicity proof 1} and the geometric inequality $2ab \leq a^2 + b^2$ to get 
	\begin{equation}\label{eq:acf monotonicity proof 4}
		\begin{aligned}
			\frac{1}{r^2} \int_{B_r} \frac{|\nabla u|^2}{|x|^{n-2}} \, dx &\leq \frac{1}{r^{n}} \int_{\partial B_r} u \partial_{\nu} u \, d\sigma + \frac{ n-2}{2r^{n+1}} \int_{\partial B_r} u^2 \, d\sigma\\
			&\leq \frac{1}{4r^{n+1}}\int_{\partial B_r} u^2 \, d\sigma + \frac{1}{r^{n-1}}\int_{\partial B_r} (\partial_{\nu} u)^2 \, d\sigma + \frac{n-2}{2r^{n+1}} \int_{\partial B_r} u^2 \, d\sigma\\
			&= \frac{2n-3}{4r^{n+1}}\int_{\partial B_r} u^2 \, d\sigma + \frac{1}{r^{n-1}}\int_{\partial B_r} (\partial_{\nu} u)^2 \, d\sigma.
		\end{aligned}
	\end{equation}
	Thus, \eqref{eq:acf monotonicity proof 3} and \eqref{eq:acf monotonicity proof 4} in \eqref{eq:acf monotonicity proof 2} give
	\begin{equation*}
		\beta'(r) \geq \frac{1}{r^{n-1}}\int_{\partial B_r} |\nabla_{\theta} u |^2 \, d\sigma - \frac{2n-3}{4r^{n+1}}\int_{\partial B_r} u^2 \, d\sigma = \frac{1}{r^{n+1}} \left(\int_{\partial B_r} u^2 \, d\sigma\right)\cdot \left(\frac{r^2 \int_{\partial B_r}|\nabla_{\theta} u|^2 \, d\sigma  }{\int_{\partial B_r} u^2 \, d\sigma} - \frac{2n-3}{4} \right).
	\end{equation*}
	After rescaling we have
	\begin{equation*}
		\frac{r^2 \int_{\partial B_r}|\nabla_{\theta} u|^2 \, d\sigma  }{\int_{\partial B_r} u^2 \, d\sigma} \geq \inf\left\{\frac{\int_{\partial B_1}|\nabla_{\theta} w|^2 \, d\sigma  }{\int_{\partial B_1} w^2 \, d\sigma}  : w\in H^{1}(\mathbb S^{n-1}) \right\}.
	\end{equation*}
	Thus, the right-hand side becomes the Rayleigh quotient for the Laplace--Beltrami operator $\Delta_{\theta}$ on the $(n-1)$-dimensional sphere $\mathbb S^{n-1}$, that is, the right-hand side coincides with the first eigenvalue of $\Delta_{\theta}$. From \cite[Chapter II, Theorem 4.1]{Shi92} we know that the first eigenvalue is $n-1$, which implies $\beta'(r) \geq 0$.
\end{proof}

\subsection{Proof of \Cref{thm:optimalreg}}\label{subsec:proof of optimal reg}
We proceed similar to the proof of \cite[Theorem 4]{AC04}. The main difference is to replace the Poincar\'{e} inequality by Maz'ya's inequality:
\begin{lemma}[Maz'ya's inequality]\label{lem:mazya}
	There exists a constant $C_M>0$, depending only on the dimension, such that for any $f\in H^1(B_r)$, $r>0$, we have
	\begin{equation}\label{eq:mazya}
		\dashint_{B_r}f^2\, dx \leq \frac{C_M}{\capacity_0(\ker f\cap B_r ; B_{2r})} \int_{B_r}|\nabla f|^2\, dx.
	\end{equation}
\end{lemma}
\begin{proof}
	The proof is analogous to \cite[Theorem 5.53]{BB11}. As the dependence on the radius $r$ is important, we include the proof. For $f\in H^1(B_r)$ we may assume that $f$ is nonnegative, else we consider $|f|$. Denote 
	\begin{equation*}
		\bar f \coloneqq  \left(\dashint_{B_r} f^2 \, dx \right)^{1\slash 2}
	\end{equation*}
	and $\varphi \coloneqq (1-f\slash \bar f)_+$. Clearly $\varphi\in H^1(B_r)$ with $\varphi \geq 1$ q.e.\ on $\ker f$. Let $\Phi \in H^1_0(B_{2r})$ be an extension of $\varphi$ such that there exists a constant $C_M>0$, independent of $r$, satisfying
	\begin{equation*}
		\norm{\nabla\Phi}_{L^2(B_{2r})}^2 \leq C_M\norm{\nabla\varphi}_{L^2(B_r)}^2.
	\end{equation*}
	Such an extension exists, for instance: Choose any linear continuous extension operator $E':H^1(B_1) \to H^1(B_2)$ and define $E \coloneqq  T^{-1} E' T$, where $Tg(x) \coloneqq  g(rx)$. We conclude
	\begin{equation*}
		\capacity_0(\ker f \cap B_r;B_{2r}) \leq \int_{B_{2r}} |\nabla\Phi|^2 \, dx \leq C_M\int_{B_{r}} |\nabla\varphi|^2 \, dx \leq C_M \left(\bar f\right)^{-2} \int_{B_r} |\nabla f|^2 \, dx.
	\end{equation*}
\end{proof}

\begin{lemma}\label{lem:growth estimate capacity}
	Let $F\subset B_1\cap \{x_n = 0\}$ be relatively closed in $B_1$ and assume that $u\in H^1(B_1)$ is a solution to the $\mathcal K_{u, 0, F}(B_1)$-obstacle problem. Then, there exists a constant $C>0$ such that for any $r\in (0, 1)$ we have
	\begin{equation}\label{eq:bound on u}
		\sup_{B_{r\slash 2}}|u| \leq C \frac{r^{(n-1)/2}}{\capacity_{0}(\Lambda(u) \cap B_r; B_{2r} )^{1\slash 2}}.
	\end{equation}
\end{lemma}
\begin{proof}
    First, it is standard to show that $u_{\pm}$, the positive and negative part of $u$, are subharmonic in $B_1$. Therefore, we apply the mean value property for subharmonic functions together with Jensen's inequality and Maz'ya's inequality \eqref{eq:mazya} to obtain
	\begin{equation*}
		\sup_{B_{r\slash 2}} u_\pm \leq \left(\dashint_{B_{r}} u^2 \, dx\right)^{1\slash 2} \leq \frac{C_M^{1\slash 2}}{\capacity_0(\ker u \cap B_r ; B_{2r})^{1\slash 2}} \left(\int_{B_r} |\nabla u|^2 \, dx \right)^{1\slash 2}.
	\end{equation*}
	Using the monotonicity of $\capacity_0$ and the inclusion $\Lambda(u) \subset \ker u$ together with \Cref{lem:acf frequency} and
	\begin{equation*}
		\int_{B_r} |\nabla u|^2  \, dx\leq r^{n-1} \beta(r),
	\end{equation*}
	 we conclude \eqref{eq:bound on u} with constant $C = C_M^{1\slash 2} \beta(1\slash 2)^{1\slash 2}$.
\end{proof}

\begin{proof}[Proof of \Cref{thm:optimalreg}]
	Without loss of generality, we may assume that $x_0 = 0 \in \Lambda(u)$. Then, \eqref{eq:bound on u} and the capacity density condition \eqref{eq:CDC contact set} of $\Lambda(u)$ imply the existence of a constant $C>0$ such that for every $x\in B_r$, $r\in(0, r_0)$:
	\begin{equation*}
		|u(0) - u(x)| = |u(x)| \leq C |x-0|^{1\slash 2}.
	\end{equation*}
	Thus, $u$ is $C^{1\slash 2}$-regular at $x_0$.
\end{proof}

\begin{remark}\label{rem:capacity converging to 0}
    From the proof of \Cref{thm:optimalreg}, in particular due to \eqref{eq:bound on u}, it is immediate that if one replaces \eqref{eq:CDC contact set} in \Cref{thm:optimalreg} with the condition: there exists $c_0, r_0 >0$ and $\alpha \in [0, 1)$ such that
    \begin{equation*}
        \capacity_0(\Lambda(u)\cap B_r(x_0) ; B_{2r}(x_0)) \geq r^{n-2 + \alpha}c_0 \quad\text{for every $r\in (0, r_0)$,}
    \end{equation*}
    then $u$ is $C^{(1-\alpha)\slash 2}$-regular at $x_0$.
\end{remark}

\subsection{Generalization for non-zero obstacles}
One can generalize \Cref{thm:optimalreg} to include non-zero obstacles as follows:
\begin{theorem}\label{thm:optimalreg non zero}
	Let $n\geq 2$ and $F\subset \{x\in B_1 \mid x_n = 0\}$ be relatively closed in $B_1$. Assume that $u\in H^1(B_1)$ is a solution to the $\mathcal K_{u, \psi, F}(B_1)$-obstacle problem for an obstacle $\psi:F\to\overline{\mathbb R}$. Moreover, assume there exists an extension $\Psi \in H^1(B_1)$ of $\psi$ satisfying:
	\begin{enumerate}
		\item For $x_0 \in \Lambda(u)$ there exists a constant $C>0$ and $\alpha \in (0, 1)$ such that $\sup_{B_r}|\Psi - \Psi(x_0)| \leq C r^\alpha$.
		\item $\Delta\Psi$ exists and satisfies
		\begin{equation*}
			\int_{B_1(x_0)} \frac{|\Delta\Psi(x)|}{|x-x_0|^{n-1}} \, dx < +\infty.
		\end{equation*}
	\end{enumerate}
	If $\Lambda(u)$ satisfies a capacity density condition at $x_0\in \Lambda(u)$, that is, there exists a constant $c_0 > 0$ and a radius $r_0>0$ such that
	\begin{equation*}
		\capacity_0(\Lambda(u) \cap B_r(x_0); B_{2r}(x_0)) \geq r^{n-2}c_0 \quad\text{for every $r\in(0, r_0)$,}
	\end{equation*}
	then $u$ is $C^{\beta}$-regular at $x_0$ for $\beta = \min\{1\slash 2, \alpha\}$.
\end{theorem}

\begin{proof}[Sketch of the proof]
    Without loss of generality, it is enough to consider $x_0 = 0$ and $\psi(0) = 0$. If $u$ solves the homogeneous obstacle problem, then $v\coloneqq u - \Psi$ satisfies the Euler-Lagrange equations
    \begin{equation}\label{eq:euler lagrange inhom}
	\begin{cases}
		-\Delta v \geq f &\text{in $B_1$} \\
		-\Delta v = 0 &\text{in $\Omega\setminus \Lambda(v)$} \\
		v\geq 0 &\text{on $F$,}
	\end{cases}
\end{equation}
    where $f \coloneqq \Delta\Psi$ and $\Lambda(v) \coloneqq \{x\in F\mid v(x) = 0\} = \Lambda(u)$. Now, one proves a variant of the monotonicity formula in \Cref{lem:acf frequency} similar to \cite{CSS08}. More precisely, one shows the monotonicity of
    \begin{equation*}
        \beta^f(r) \coloneqq \beta(r, v) + \int_0^r \frac{1}{s^2} \int_{B_s} \frac{ {f v} }{|x|^{n-2}} \, dx \, ds \quad\text{for $r\in (0, 1)$}.
    \end{equation*}    
    Note, for the second term we have
    \begin{equation*}
        \int_0^r \frac{1}{s^2} \int_{B_s} \frac{f v}{|x|^{n-2}} \, dx \, ds \leq \sup_{B_r}|v| \cdot \int_{B_1} \frac{|\Delta \Psi(x)|}{|x|^{n-1}} \, dx <+\infty
    \end{equation*}
    due to the condition $(ii)$. The monotonicity of $\beta^f$ follows similarly as in \Cref{lem:acf frequency} with the help of a variant of \Cref{lem:integration by parts u}:
    \begin{equation*}
		\int_{B_r} \nabla v \cdot \nabla \varphi \, dx = \int_{B_r} f \varphi \, dx + \int_{\partial B_r} \varphi \partial_{\nu} v \, d\sigma \quad\text{for every $\varphi\in H^1(B_1)$ such that $\varphi = 0$ q.e.\ on $\Lambda(u)$.}
	\end{equation*}
    To finish the proof, we use the subharmonicity of $u_{\pm}$ as in \Cref{lem:growth estimate capacity} but incorporate $\Psi$ as follows:
    \begin{equation*}
		\sup_{B_r} u_{\pm} \leq \dashint_{B_r} u_{\pm} \, dx \leq \dashint_{B_r} |v| \, dx + \dashint_{B_r} |\Psi| \, dx \leq \left(\dashint_{B_r} |v|^2 \, dx\right)^{\frac 12} + \dashint_{B_r} |\Psi| \, dx.
	\end{equation*}
    After using Maz'ya's inequality \eqref{eq:mazya} on the integral with respect to $v$, we use the estimate
    \begin{equation*}
    \begin{aligned}
        \int_{B_r}|\nabla v|^2 \, dx &\leq r^{n-1} \beta(r, v) =  r^{n-1} \left(\beta^f(r) - \int_0^r \frac{1}{s^2} \int_{B_s} \frac{ {f v} }{|x|^{n-2}} \, dx \, ds\right) \\
        &\leq r^{n-1}\left( \beta^f(1\slash 2) + \sup_{B_r}|v| \cdot \int_{B_1} \frac{|\Delta \Psi(x)|}{|x|^{n-1}} \, dx \right),
    \end{aligned}  
    \end{equation*}
    for any $r\in(0, 1\slash 2)$. Hence, we conclude
    \begin{equation*}
        \sup_{B_r} u_{\pm} \leq C(r^{1\slash 2} + r^{\alpha}),
    \end{equation*}
    where the constant depends on $c_0$, $C_M$, $\beta(1\slash 2, u)$, $\sup_{B_{3\slash 4}} |u|$, $(i)$, and $(ii)$.
\end{proof}

\begin{remark}
    Let us comment on the conditions $(i)$ and $(ii)$ in \Cref{thm:optimalreg non zero}. It is easy to see that both are satisfied if $\Psi \in C^{1, 1}(B_1)$. In fact, $\Psi\in W^{2, n+\delta}(B_1)$, for any $\delta > 0$, is sufficient. Note, by Sobolev-Morrey embedding, this implies that $\Psi\in C^{1,\alpha}(B_1)$ for $\alpha = 1-\frac{n}{n+\delta}$. In particular, in this case $\Psi$ is Lipschitz continuous which gives the optimal $C^{1\slash 2}$-regularity in \Cref{thm:optimalreg non zero} as $(i)$ is satisfied for $\alpha = 1$.

    Moreover, we expect that $(ii)$ implies $(i)$ for $\alpha = 1$, thus $(ii)$ implies the optimal $C^{1\slash 2}$-regularity already. The reason is that the integral in $(ii)$ coincides with the first order Riesz potential $I_1(|\Delta \Psi|\mathds 1_{B_1(x_0)})$ at $x_0$, see \cite{AH96}. A simple calculation shows if $w = \Phi \ast f$, for the fundamental solution $\Phi$ of the Laplacian on $\mathbb R^n$, then $|\nabla w| \lesssim I_1(f)$. Thus, roughly speaking, boundedness of $I_1(f)$ implies Lipschitz continuity of $w$. However, it is not clear how to utilize this observation in a rigorous manner, especially since the statements in $(i)$ and $(ii)$ are only pointwise in $x_0$.
\end{remark}

\section{Almost optimal H\"{o}lder regularity in the half-hyperplane obstacle case}\label{sec-almostoptimal}
In this section we are concerned with Signorini-type problems in the unit ball $B_1 \subset \mathbb R^n$ and the obstacle region $F = \{(x', x_{n-1}, 0) \mid x' \in \mathbb{R}^{n-2}, \ x_{n-1}\leq 0\} \cap B_1$ being the left half of the horizontal hyperplane in $B_1$. We recall that, in order to investigate regularity of solutions to (classical) Signorini problems, it is useful to take the associated Dirichlet boundary value problems into consideration; see \cite{PSU12, FR16, Sal12, FR22} for instance. Nevertheless, since the obstacle condition is achieved only on a subset $F$ of $\{x_n=0\}$, it turns out that the behavior of solutions to Signorini-type problems can be described by the associated \textit{mixed boundary value problems} (or mixed BVP for short). Therefore, before we establish the suitable H\"older regularity for solutions of Signorini-type problems, we develop the related mixed BVP theory.

\subsection{Well posedness of Mixed BVPs}
Let $\Omega$ be a Lipschitz domain in $\mathbb{R}^n$ and $\partial \Omega=\Gamma_D \cup \Gamma_N$, where $\Gamma_D$ and $\Gamma_N$ denote the Dirichlet boundary part and the Neumann boundary part of $\partial \Omega$, respectively. As $\Omega$ is a Lipschitz domain, it is an $H^1$-extension domain. Thus, for any function $u\in H^1(\Omega)$ there exists a unique quasi-continuous extension $\widetilde u : \overline{\Omega} \to \mathbb R$ where its boundary values are unique in a quasi-everywhere sense with respect to $\capacity_1(\cdot ; \mathbb R^n)$. In particular, we say $\Gamma \subset\partial\Omega$ is of \emph{positive capacity} if $\capacity_1(\Gamma; \mathbb R^n) > 0$. Hence, statements like \enquote{$u = \psi$ q.e.\ on $\Gamma_D$} are well stated in the aforementioned sense. 

Then, we are interested in a mixed BVP given by
\begin{equation}\label{eq:mixedBVP}
	\left\{
	\begin{aligned}
		-\Delta v&=f && \text{in $\Omega$}\\
		v&= \psi && \text{on $\Gamma_D$}\\
		\partial_{\nu} v&=0 && \text{on $\Gamma_N$},
	\end{aligned}\right.
\end{equation}
where $f \in L^2(\Omega)$, $\psi \in H^1(\Omega)$, and $\nu$ is the outward unit normal vector on $\Gamma_N$.

\begin{definition}
	Let $\Omega\subset\mathbb R^n$ be a bounded Lipschitz domain with $\partial \Omega=\Gamma_D \cup \Gamma_N$, $f \in L^2(\Omega)$, and $\psi \in H^1(\Omega)$. A function $v\in H^1(\Omega)$ is said to be a \emph{weak solution} to the mixed BVP \eqref{eq:mixedBVP}, if for every $\varphi\in H^1_{\Gamma_D}(\Omega) \coloneqq  \{\varphi\in H^1(\Omega) \mid \varphi = 0 \text{ q.e.\ on }\Gamma_D \}$, we have
	\begin{equation*}
		\int_{\Omega} \nabla v \cdot \nabla \varphi \, dx = (f, \varphi)_{L^2(\Omega)}
	\end{equation*}
	and $v = \psi$ q.e.\ on $\Gamma_D$.
\end{definition}

The well posedness of \eqref{eq:mixedBVP} follows from the standard variational argument; see also \cite[Theorem 8]{BL24} for a similar result.
\begin{theorem}\label{thm:wellposedness}
	Let $\Omega\subset\mathbb R^n$ be a bounded Lipschitz domain with $\partial \Omega=\Gamma_D \cup \Gamma_N$. Moreover, suppose that $\Gamma_D$ is closed and of positive capacity, $f \in L^2(\Omega)$, and $\psi \in H^1(\Omega)$. Then, there exists a unique weak solution $v\in H^1(\Omega)$ of \eqref{eq:mixedBVP}. Moreover, there exists a universal constant $C>0$ such that
    \begin{equation*}
        \norm{v}_{H^1(\Omega)} \leq C\left(\norm{\psi}_{H^1(\Omega)}+\|f\|_{L^2(\Omega)}\right).
    \end{equation*}
\end{theorem}
\begin{proof}
	We set the linear functional
	\begin{equation*}
		-\Delta \psi (\varphi) \coloneqq  \int_{\Omega} \nabla \psi \cdot \nabla \varphi \, dx \quad\text{for $\varphi\in H^1(\Omega)$}.
	\end{equation*}
	It follows from the Lax--Milgram theorem, applied to the usual Dirichlet energy on $H^1_{\Gamma_D}(\Omega)$, that there exists a unique $w\in H^1_{\Gamma_D}(\Omega)$ satisfying
	\begin{equation*}
		\int_{\Omega} \nabla w \cdot \nabla \varphi \, dx = \Delta \psi(\varphi)+(f, \varphi)_{L^2(\Omega)} \quad\text{for every $\varphi\in H^1_{\Gamma_D}(\Omega)$}
	\end{equation*}
	with the uniform bound
	\begin{equation*}
		\norm{w}_{H^1(\Omega)} \leq C\left(\norm{\psi}_{H^1(\Omega)}+\|f\|_{L^2(\Omega)}\right).
	\end{equation*}
	Now our solution can be found by $v\coloneqq w + \psi \in H^1(\Omega)$ with the desired uniform estimate.
\end{proof}

We also provide the maximum principle and the comparison principle for weak solutions to \eqref{eq:mixedBVP}.
\begin{theorem}[Maximum principle for subsolutions]
	Let $\Omega\subset\mathbb R^n$ be a bounded Lipschitz domain with $\partial \Omega=\Gamma_D \cup \Gamma_N$. Moreover, suppose that $\Gamma_D$ is closed and of positive capacity. Let $v\in H^1(\Omega)$ be a weak solution to
	\begin{equation*}
	\left\{\begin{aligned}
		-\Delta v &\leq 0 &&\text{in $\Omega$}\\
		v &\leq 0 &&\text{in $\Gamma_D$}\\
		\partial_{\nu} v &\leq 0 &&\text{in $\Gamma_N$},
	\end{aligned}\right.
	\end{equation*}
	that means, for every nonnegative $\varphi\in H^1_{\Gamma_D}(\Omega)$ we have
	\begin{equation*}
		\int_{\Omega} \nabla v \cdot \nabla \varphi \, dx \leq 0
	\end{equation*}
	and $v\leq 0$ q.e.\ on $\Gamma_D$. Then, $v\leq 0$ a.e.\ in $\Omega$.
\end{theorem}
\begin{proof}
	Let us write $v = v_+ - v_-$, where $v_+=\max\{v, 0\}$ and $v_-=-\min\{v, 0\}$. Since $v_+\in H^1(\Omega)$, $v_+ = 0$ q.e.\  on $\Gamma_D$, and $v_+ \geq 0$, we can take $v_+$ as a test function to obtain
	\begin{equation*}
		\int_{\Omega} \nabla v_+ \cdot \nabla v_+ \, dx - \int_{\Omega} \nabla v_- \cdot \nabla v_+ \, dx=\int_{\Omega} \nabla v \cdot \nabla v_+ \, dx  \leq 0.
	\end{equation*}
	As $v_- = 0$ on the support of $u_+$, the strong locality of the Dirichlet energy implies that $(\nabla v_-, \nabla v_+)_{L^2(\Omega)} = 0$. Therefore, $\norm{\nabla v_+}_{L^2(\Omega)} \leq 0$, which implies that $v_+ = \const = 0$.
\end{proof}

\begin{corollary}[Comparison principle] \label{cor comparison principle mbvp}
	Let $\Omega\subset\mathbb R^n$ be a bounded Lipschitz domain with $\partial \Omega=\Gamma_D \cup \Gamma_N$. Moreover, suppose that $\Gamma_D$ is closed and of positive capacity. Let $v_1, v_2 \in H^1(\Omega)$ satisfy
	\begin{equation*}
		\left\{\begin{aligned}
				-\Delta v_1 &\leq -\Delta v_2 &&\text{in $\Omega$}\\
			v_1 &\leq v_2 &&\text{in $\Gamma_D$}\\
			\partial_{\nu} v_1 &\leq \partial_{\nu} v_2 &&\text{in $\Gamma_N$},
		\end{aligned}\right.
	\end{equation*}
	which means, for every nonnegative $\varphi\in H^1_{\Gamma_D}(\Omega)$ we have
	\begin{equation*}
		\int_{\Omega} \nabla v_1 \cdot \nabla \varphi \, dx \leq \int_{\Omega} \nabla v_2 \cdot \nabla \varphi \, dx
	\end{equation*}
	and $v_1\leq v_2$ q.e.\ on $\Gamma_D$. Then, $v_1 \leq v_2$ a.e.\ in $\Omega$.
\end{corollary}

\subsection{Almost optimal H\"older regularity for mixed BVPs}\label{subsec-mixedbvp-optimal}
We now restrict our attention to the special case that is related to our situation of Signorini-type problems. That means, we fix $\Omega=B_1^+ \subset \mathbb{R}^n$ with $n\geq 2$, $\Gamma_N=\{(x', x_{n-1}, 0): x_{n-1}>0\} \cap B_1$, and $\Gamma_D=\partial B_1^+ \setminus \Gamma_N=F \cup (\partial B_1)^+$ for $F\coloneqq \{(x', x_{n-1},0) : x_{n-1} \leq 0\} \cap B_1$. Here we denote $x=(x', x_{n-1}, x_n) \in \mathbb{R}^n$ and use the cylindrical coordinates $x' \in \mathbb{R}^{n-2}$, $r=\sqrt{x_{n-1}^2+x_n^2}$, and
\begin{equation*}
	\theta=\theta(x_{n-1}, x_n)=
	\begin{cases}
		\arccos (x_{n-1}/r) & \text{if $x_n \geq 0$}\\
		2\pi-\arccos (x_{n-1}/r) & \text{if $x_n < 0$}.
	\end{cases}
\end{equation*}
We also assume that the Dirichlet data $\psi$ is zero on the horizontal hyperplane $\{x_n = 0\}$. Then, we are concerned with a weak solution $v \in H^1(B_1^+) \cap C(\overline{B_1^+})$ to the mixed BVP
\begin{equation}\label{eq:simple}
		\left\{
		\begin{aligned}
			-\Delta v &=f && \text{in $B_1^+$}\\
			v &=0 && \text{on $F$}\\
			\partial_{\nu} v &= 0 && \text{on $\Gamma_N$}.
		\end{aligned}\right.
\end{equation}
Here $\vec\nu=(0, \cdots, 0, -1)$ on $\Gamma_N$ and $f\in H^1(B_1^+)\cap L^{\infty}(B_1^+)$. 

Let us make a couple of comments on solutions to \eqref{eq:simple}. First of all, since we are interested in the regularity of $v$ in $\overline{B_{1/2}^+}$, we do not specify a boundary behavior of $v$ on $(\partial B_1)^+$. Next, we can show that $v$ is continuous in $\overline{B_{1\slash 2}^+}$ as follows: By standard regularity theory $v$ is continuous in the non-intersection region $\overline{B_{1\slash 2}^+}\setminus (\Gamma_D\cap \overline{\Gamma_N})$, see \cite[Chapter 6]{GT01} for instance. Moreover, $v$ is continuous at $x_0 \in I_{1/2}\coloneqq \Gamma_D \cap \overline{\Gamma_N} \cap \overline{B_{1\slash 2}}=\{(z', 0, 0) : |z'| \leq 1/2\}$ by using the barrier constructed in \cite{Lie86, Lie89}, since the straight $\Gamma_N$-wedge condition (defined in \cite{Lie86}) is satisfied at $x_0 \in I_{1/2}$. We note that this condition does not hold when we reach the boundary $\partial B_1$, for example at $x_0=(\pm 1, \cdots, 0, 0) \in \partial B_1\cap\Gamma_D \cap \overline{\Gamma_N}$ when $n \geq 3$. Finally, we would like to point out that the assumption  $v \in C(\overline{B_1^+})$ is not restrictive. In fact,  a solution  $v$ to \eqref{eq:simple} is continuous and so bounded in $\overline{B_{r}^+}$ for any $r \in (0,1)$, by repeating the previous argument. Therefore, we may consider a slightly smaller ball than $B_1$, if necessary.

Our goal in this subsection is to derive the almost optimal H\"older regularity of solutions in this special case, that is, we will show that $v \in C^{1/2-\varepsilon}(\overline{B_{1/4}^+})$ for any $\varepsilon \in (0, 1/2)$. Let us first prove that the pointwise regularity of $v$ at the intersection region $I_{1/2}$ is $C^{1/2-\varepsilon}$ for any $\varepsilon \in (0, 1/2)$. 
\begin{lemma}[Intersection points]\label{lem-optimal1}
	Let $\varepsilon \in (0, 1/2)$. Let $v \in H^1(B_1^+) \cap C(\overline{B_1^+})$ be a weak solution to \eqref{eq:simple}, where $f\in H^1(B_1^+)\cap L^{\infty}(B_1^+)$.
	 Then, $v$ is $C^{1/2-\varepsilon}$ at $z \in I_{1/2}$; that is, there exists a constant $C>0$ depending only on $n$ and $\varepsilon$ such that
	\begin{equation*}
		|v(x)| \leq C \left(\|v\|_{L^{\infty}(B_1^+)}+\|f\|_{L^{\infty}(B_1^+)} \right)  |x-z|^{1/2-\varepsilon}
	\end{equation*}
for any $x=(x', x_{n-1}, x_n) \in B_{1/2}^+(z)$.
\end{lemma}

\begin{proof}
    We may assume that $z=0$ by applying the standard translation and scaling argument, if necessary. We define an auxiliary function $h$ given by 
	\begin{equation*}
		h(x', x_{n-1}, x_n)\coloneqq r^{1/2-\varepsilon} \cos\left[\left(\frac{1-\varepsilon}{2}\right)\theta\right] \quad \text{for $x_n \geq 0$}.
	\end{equation*}
	Then, a direct calculation shows that
	\begin{equation*}
		\left\{
		\begin{aligned}
			-\Delta h &= \left(\frac{\varepsilon}{2}-\frac{3\varepsilon^2}{4}\right)r^{-3/2-\varepsilon} \cos\left[\left(\frac{1-\varepsilon}{2}\right)\theta\right]  && \text{in $B_1^+$}\\
		  h &= r^{1/2-\varepsilon} \sin(\varepsilon \pi/2) \geq 0 && \text{on $F$}\\
			\partial_{\nu} h&=\partial_{\theta}h=0 && \text{on $\Gamma_N$},
		\end{aligned}\right.
	\end{equation*} 
	and $r^{1/2-\varepsilon} \sin (\varepsilon \pi/2) \leq h(x) \leq r^{1/2-\varepsilon}$ for $\theta \in [0, \pi]$. We note that  $\varepsilon/2-3\varepsilon^2/4>0$ for any $\varepsilon \in (0, 1/2)$.
	
	We now construct a barrier function (supersolution)
	\begin{equation*}
		w(x)\coloneqq \left(\|v\|_{L^{\infty}(B_1^+)}+{\|f\|_{L^{\infty}(B_1^+)}}\right) \left(C_1h(x)+C_2|x|^2\right)
	\end{equation*}
	for some $C_1, C_2>0$. Then, there exist constants $C_1, C_2>0$ depending only on $n$ and $\varepsilon$ such that
	\begin{equation*}
		\left\{
		\begin{aligned}
			-\Delta w &\geq \|f\|_{L^{\infty}(B_1^+)}&& \text{in $B_1^+$}\\
			w &\geq 0 && \text{on $F$}\\
            w &\geq \|v\|_{L^{\infty}(B_1^+)} &&\text{on $(\partial B_1)^+$}\\
			\partial_{\nu} w&=0 && \text{on $\Gamma_N$}.
		\end{aligned}\right.
	\end{equation*} 
	Therefore, an application of the comparison principle (\Cref{cor comparison principle mbvp}) between $v$ and $w$ yields that $v \leq w$ in $B_1^+$. In particular, we obtain
	\begin{equation*}
		v(x) \leq w(x) \leq C\left(\|v\|_{L^{\infty}(B_1^+)}+{\|f\|_{L^{\infty}(B_1^+)}}\right)|x|^{1/2-\varepsilon},
	\end{equation*}
	where $C>0$ depends only on $n$ and $\varepsilon$.
	By considering $-v$, $-f$ instead of $v$, $f$, we arrive at the desired estimate.
\end{proof}

We next extend this pointwise regularity to any point in $B_{1/4}'=\{(x', x_{n-1}, 0) \in B_{1/4}\}$.

\begin{lemma}[Pointwise boundary regularity]\label{lem-ptbdryest}
    Let $\varepsilon \in (0, 1/2)$. Let $v \in H^1(B_1^+) \cap C(\overline{B_1^+})$ be a weak solution to
	\eqref{eq:simple}, where $f\in H^1(B_1^+)\cap L^{\infty}(B_1^+)$. Then, $v$ is $C^{1/2-\varepsilon}$ at $z \in B_{1/4}'$; that is, there exists a constant $C>0$ depending only on $n$ and $\varepsilon$ such that
	\begin{equation*}
		|v(x)-v(z)| \leq C \left(\|v\|_{L^{\infty}(B_1^+)}+\|f\|_{L^{\infty}(B_1^+)} \right)  |x-z|^{1/2-\varepsilon}
	\end{equation*}
for any $x=(x', x_{n-1}, x_n) \in B_{1/4}^+(z)$.
\end{lemma}

\begin{proof}
    If $z_{n-1}=0$, then the desired estimate follows from \Cref{lem-optimal1}. We assume that $z_{n-1}>0$, that is, $z \in \Gamma_N$ with $r\coloneqq z_{n-1} \in (0,  1/4)$. By applying \Cref{lem-optimal1} for $(z', 0, 0)$, we observe that
    \begin{equation}\label{eq-bdryest}
        |v(x)| \leq C\left(\|v\|_{L^{\infty}(B_1^+)}+\|f\|_{L^{\infty}(B_1^+)} \right)  |x-(z',0,0)|^{1/2-\varepsilon} \quad \text{for $x \in B_{1/2}^+(z', 0,0)$}.
    \end{equation}
    We now choose $x \in B_{1/4}^+(z)$ and split two cases as follows.
    \begin{enumerate}
        \item ($|x-z| \leq r/2$) In $B_{3r/4}^+(z)$, $v$ solves the Poisson equation with zero Neumann boundary condition. We thus apply the boundary estimate developed in \cite[Section 6.7]{GT01} to obtain
        \begin{equation*}
            |v(x)-v(z)| \leq C\left(\frac{\|v\|_{L^{\infty}(B_{3r/4}^+(z))}}{r^{1/2-\varepsilon}}+r^{2-(1/2-\varepsilon)}\|f\|_{L^{\infty}(B_1^+)} \right)|x-z|^{1/2-\varepsilon}.
        \end{equation*}
        On the other hand, \eqref{eq-bdryest} together with the triangle inequality gives that
        \begin{equation*}
            \|v\|_{L^{\infty}(B_{3r/4}^+(z))} \leq C\left(\|v\|_{L^{\infty}(B_1^+)}+\|f\|_{L^{\infty}(B_1^+)} \right) r^{1/2-\varepsilon},
        \end{equation*}
        which implies that
        \begin{equation*}
            |v(x)-v(z)| \leq C\left(\|v\|_{L^{\infty}(B_1^+)}+\|f\|_{L^{\infty}(B_1^+)} \right) |x-z|^{1/2-\varepsilon}.
        \end{equation*}

        \item ($|x-z| >r/2$) In this case, \eqref{eq-bdryest} again shows that
        \begin{equation*}
        \begin{aligned}
            |v(x)-v(z)| \leq |v(x)|+|v(z)| &\leq C\left(\|v\|_{L^{\infty}(B_1^+)}+\|f\|_{L^{\infty}(B_1^+)} \right) \left(|x-z|^{1/2-\varepsilon}+r^{1/2-\varepsilon}\right)\\
            &\leq C\left(\|v\|_{L^{\infty}(B_1^+)}+\|f\|_{L^{\infty}(B_1^+)} \right) |x-z|^{1/2-\varepsilon}.
        \end{aligned}
        \end{equation*}
    \end{enumerate}
    The other case $z_{n-1}<0$ follows from the same argument; just replace the Neumann problem by the Dirichlet problem.    
\end{proof}

We finally combine the interior regularity and the pointwise boundary regularity by following the argument in \cite[Theorem 3.1]{Wan92}.
\begin{theorem}[Global regularity]\label{thm-almostoptimal-mixedBVP}
    Let $\varepsilon \in (0, 1/2)$. Let $v \in H^1(B_1^+) \cap C(\overline{B_1^+})$ be a weak solution to
	\eqref{eq:simple}, where $f\in H^1(B_1^+)\cap L^{\infty}(B_1^+)$. Then, $v \in C^{1/2-\varepsilon}(\overline{B_{1/4}^+})$; that is, there exists a constant $C>0$ depending only on $n$ and $\varepsilon$ such that
	\begin{equation*}
		|v(x)-v(y)| \leq C \left(\|v\|_{L^{\infty}(B_1^+)}+\|f\|_{L^{\infty}(B_1^+)} \right)  |x-y|^{1/2-\varepsilon}
	\end{equation*}
for any $x, y \in B_{1/4}^+$.
\end{theorem}

\begin{proof}
    Since the proof is essentially the same as in \Cref{lem-ptbdryest}, we only overview the proof. If $x_n=0$ or $y_n=0$, then it is a consequence of \Cref{lem-ptbdryest}. Otherwise, we let $\hat{x}\coloneqq (x', x_{n-1}, 0) \in B_{1/4}'$ with $r=x_n \in (0, 1/4)$. We utilize the estimate obtained in \Cref{lem-ptbdryest} for $\hat{x}$ and split two cases depending on the distance between $x$ and $y$. In particular, if $|x-z|<r/2$, we recall the interior regularity of solutions to the Poisson equation; see \cite[Proposition 4.10]{CC95} for instance.
\end{proof}

\subsection{Almost optimal regularity for Signorini-type problems}
We now consider the special situation provided in \Cref{subsec-mixedbvp-optimal} (with $f \equiv 0$) and claim that a weak solution $u$ of the associated Signorini-type problem belongs to $C^{1/2-\varepsilon}(\overline{B_{1/4}^+})$ for any $\varepsilon>0$. In the following theorem, we may assume $u \in L^{\infty}(B_1)$ by repeating a similar argument illustrated at the beginning of \Cref{subsec-mixedbvp-optimal}.
\begin{theorem}[Almost optimal regularity]\label{thm-almostoptimal}
	 Let $\Omega=B_1 \subset \mathbb{R}^n$ with $n\geq 2$ and $F=\{(x', x_{n-1},0) : x_{n-1}\leq 0\} \cap B_1$. If $u$ is a solution to the $\mathcal{K}_{u, 0, F}(B_1)$-obstacle problem, then $u \in C^{1/2-\varepsilon}(\overline{B_{1/4}})$ for any $\varepsilon>0$. Moreover, there exists a constant $C>0$ depending only on $n$ and $\varepsilon$ such that
     \begin{equation*}
        \|u\|_{C^{1/2-\varepsilon}(\overline{B_{1/4}})} \leq C\|u\|_{L^{\infty}(B_{1})}.
    \end{equation*}
\end{theorem}

\begin{proof}
    We first assume that $u$ is symmetric with respect to the hyperplane $\{x_n=0\}$. We will construct an extended obstacle $\overline{\Psi}$ defined in $B_{3/4}$ and understand $u$ as a solution to the classical obstacle problem with the obstacle $\overline{\Psi}$. For this purpose, we let $\eta \in C_c^{\infty}(B_1)$ be a cutoff function such that $\eta=1$ in $B_{3/4}$ and $0 \leq \eta \leq 1$. Then, \Cref{thm:wellposedness} guarantees the existence of a unique weak solution $\Psi \in H^1(B_1^+)$ of the following mixed BVP:
    \begin{equation*}
        \left\{
        \begin{aligned}
            -\Delta \Psi&=0 && \text{in $B_1^+$}\\
            \Psi&= -\|u\|_{L^{\infty}(B_1)}\cdot(1-\eta) && \text{on $\Gamma_D=F \cup (\partial B_1)^+$}\\
            \partial_{\nu}\Psi&= 0 && \text{on $\Gamma_N$}.
        \end{aligned}
        \right.
    \end{equation*}
    We observe that $\Psi=0$ on $F \cap B_{3/4}$ and by the comparison principle \Cref{cor comparison principle mbvp},
    \begin{equation*}
        -\|u\|_{L^{\infty}(B_1)} \leq \Psi \leq u \quad \text{in $B_1^+$}.
    \end{equation*} 
    In particular, $u$ can be understood as a solution to the classical $\mathcal{K}_{u, \overline{\Psi}, B_{3/4}}(B_{3/4})$-obstacle problem, where 
    \begin{equation*}
        \overline{\Psi}(x', x_n) \coloneqq 
        \begin{cases}
            \Psi(x', x_n) & \text{if $x_n \geq 0$}\\
            \Psi(x', -x_n) & \text{if $x_n<0$}.
        \end{cases}
    \end{equation*}
    Since the almost optimal estimate \Cref{thm-almostoptimal-mixedBVP} for $\Psi$ in $B_{3/4}^+$ yields that
    \begin{equation*}
        \|\Psi\|_{C^{1/2-\varepsilon}(\overline{B_{1/4}^+})} \leq C\|u\|_{L^{\infty}(B_1)},
    \end{equation*}
    the desired estimate follows from \cite[Theorem 2]{Caf98}.
    
    In the case that $u$ is not symmetric with respect to $\{x_n=0\}$, we let $w \in H^1(B_1)$ be the weak solution to the Dirichlet boundary value problem 
    \begin{equation*}
		\left\{\begin{aligned}
			-\Delta w &= 0 &&\text{in $B_1$}\\
			w &= g &&\text{on $\partial B_1$,}
		\end{aligned}\right.
	\end{equation*}
	where $g(x, y) = \left(u(x', x_n) - u(x', -x_n) \right)\slash 2$. Then, it is easy to check that $\bar{u} \coloneqq u-w$ becomes a symmetric solution to the $\mathcal{K}_{\bar{u}, 0, F}(B_1)$-obstacle problem, which finishes the proof.
\end{proof}

\begin{remark}\label{rem:AC remark}
    In the articles \cite{ACM22, AC22, AC23} the authors proved a Boyarski--Meyer estimate for solutions to mixed boundary value problems of the form
    \begin{equation}
	\left\{
	\begin{aligned}\label{eq-mixedbvp-ACM}
		-\mathcal L u&= \div H && \text{in $\Omega$}\\
		u&= 0 && \text{on $\Gamma_D$}\\
		\partial_{\nu} u&=0 && \text{on $\Gamma_N$},
	\end{aligned}\right.
\end{equation}
    where $\mathcal L$ is an elliptic operator in divergence form, the $p$-Laplacian or $p(\cdot)$-Laplacian, $\partial\Omega = \Gamma_D \cup \Gamma_N$, and the Dirichlet part of the boundary $\Gamma_D$ satisfies a variant of the capacity density condition \eqref{eq:CDC contact set}. The authors proved increased integrability of the gradient of the solution if $H$ has higher integrability, more precisely, if $H\in (L^{p'(\cdot)(1+\delta_0)}(\Omega))^n$, $\delta_0>0$, then $\nabla u\in (L^{p(\cdot)(1+\delta)}(\Omega))^n$ for some $\delta_0>\delta >0$. In particular, from standard embedding results, the solution to \eqref{eq-mixedbvp-ACM} in the $p$-Laplacian case is H\"{o}lder continuous if $p=n$. Thus, this result goes in a similar direction as our results due to the connection of Signorini-type problems and mixed BVPs. Although such result does not give the (almost) optimal regularity, it allows for a more general treatment of operators and obstacle\slash Dirichlet sets $F$ (including fractals) and may be compared to \Cref{thm:preliminary holder reg}.
\end{remark}

\subsection{Generalization to non-zero obstacles}
We also extend the results of \Cref{sec-almostoptimal} to include non-zero obstacles $\psi$ as follows:
\begin{theorem}\label{thm:almost optimal non zero}
    Let $\Omega = B_1 \subset \mathbb R^n$, $n\geq 2$, $F=\{(x', x_{n-1}, 0) \mid   x_{n-1} \leq  0 \} \cap B_1$, and $\psi\in C^{\beta_0}(F)$ for some $\beta_0 \in (0, 1)$. Moreover, assume there exists an extension $\Psi_0 \in H^1(B_1)$ of $\psi$. If $u\in H^1(\Omega)$ solves the $\mathcal K_{u, \psi, F}(\Omega)$-obstacle problem, then $u$ is $C^\beta$-regular at $x_0 = 0$ for some $\beta \in (0, 1\slash 2)$ such that $\beta \leq \beta_0$.
\end{theorem}

\begin{proof}[Sketch of the proof]
    This lemma follows from repeating the arguments in \Cref{sec-almostoptimal} with several modifications. We summarize the main strategy as follows. Here we may assume that $u$ is bounded in $L^{\infty}(B_1)$ and symmetric with respect to the hyperplane $\{x_n=0\}$ as usual.
    \begin{enumerate}
        \item  We extend the obstacle $\psi$ to $\Psi$ on the domain $B_{3/4}$ by solving the associated mixed BVP. Then, $u$ can be understood as the solution to the $\mathcal{K}_{u, \Psi, B_{3/4}}(B_{3/4})$-obstacle problem, that is, a classical obstacle problem in $B_{3/4}$. 

        \item We show the almost optimal H\"older regularity of $\Psi$ by constructing appropriate barrier functions.

        \item We transport this regularity to $u$ by means of \cite[Theorem 2]{Caf98}.
    \end{enumerate}
     Since the case of $n \geq 3$ can be proved in a similar way to the case of $n=2$ and the steps $(i)$, $(iii)$ do not require essential changes, we focus on the step $(ii)$ when $n=2$. In particular, we present the main difference arising from the proof of \Cref{lem-optimal1}, where the choice of H\"older exponent $\beta \in (0, 1/2) \cap (0, \beta_0]$ can be explained.

     To be precise, we let $v \in H^1(B_1^+) \cap C(\overline{B_1^+})$ be a weak solution to the mixed BVP 
    \begin{equation*}
		\left\{
		\begin{aligned}
			-\Delta v &=f && \text{in $B_1^+$}\\
			v &=\psi && \text{on $F$}\\
			\partial_{\nu} v &= 0 && \text{on $\Gamma_N$}.
		\end{aligned}\right.
\end{equation*}    
Without loss of generality, we may assume that $\psi(0)=0$. 

We choose $\varepsilon\coloneqq 1/2-\beta>0$ in the proof of \Cref{lem-optimal1} and recall an auxiliary function $h$ given by 
	\begin{equation*}
		h(x_1, x_2)\coloneqq r^{\beta} \cos\left[\left(\frac{1/2+\beta}{2}\right)\theta\right] \quad \text{for $x_2 \geq 0$ and $r=\sqrt{x_1^2+x_2^2}$},
	\end{equation*}
    so that $h(x)=C_\beta r^{\beta}$ on $F$ for some constant $C_{\beta}>0$. Moreover, we construct a modified barrier function 
	\begin{equation*}
		w(x)\coloneqq \left(\|v\|_{L^{\infty}(B_1^+)}+{\|f\|_{L^{\infty}(B_1^+)}}+\|\psi\|_{C^{\beta_0}(F)}\right) \left(C_1h(x)+C_2r^2\right)
	\end{equation*}
	for some $C_1, C_2>0$. Then, there exist constants $C_1, C_2>0$ depending only on $\beta$ such that
	\begin{equation*}
		\left\{
		\begin{aligned}
			-\Delta w &\geq -\|f\|_{L^{\infty}(B_1^+)}&& \text{in $B_1^+$}\\
			w &\geq v && \text{on $F$}\\
            w &\geq \|v\|_{L^{\infty}(B_1^+)} &&\text{on $(\partial B_1)^+$}\\
			\partial_{\nu} w&=0 && \text{on $\Gamma_N$}.
		\end{aligned}\right.
	\end{equation*}
    In fact, it follows from $\beta \leq \beta_0$ that 
    \begin{equation*}
        v(x)=\psi(x)\leq \|\psi\|_{C^{\beta}(F)} r^{\beta_0} \leq \|\psi\|_{C^{\beta}(F)} r^{\beta} \leq w(x) \quad \text{on $F$},
    \end{equation*}
    when we choose $C_1>0$ sufficiently large. Therefore, an application of the comparison principle between $v$ and $w$ yields that $v \leq w$ in $B_1^+$. In other words, we obtain
	\begin{equation*}
		v(x) \leq w(x) \leq C\left(\|v\|_{L^{\infty}(B_1^+)}+{\|f\|_{L^{\infty}(B_1^+)}}+\|\psi\|_{C^{\beta_0}(F)}\right)r^{\beta},
	\end{equation*}
	where $C>0$ depends only on $\beta$. By considering $-v$, $-f$ instead of $v$, $f$, we arrive at the desired estimate.

    We deduce the desired conclusion after repeating the remaining part of \Cref{sec-almostoptimal}.
\end{proof}

\section{Optimal H\"{o}lder regularity in the half-line obstacle case in two dimensions}\label{sec:opt holder reg}
In this section we focus on the two dimensional Signorini-type problem in the unit ball $B_1 \subset \mathbb R^2$ and obstacle region $F = \{(x_1, 0) \mid x_1\leq 0\} \cap B_1$ being the left half of the horizontal line in $B_1$. Compared to \Cref{sec-almostoptimal}, we restrict our attention to $n=2$ but develop the optimal $C^{1/2}$-regularity in \Cref{thm:optimal holder reg half line} and provide blowup profiles in \Cref{lem-classification}.

Let $u\in H^1(B_1)$ be a solution to the $\mathcal K_{u, 0, F}(B_1)$-obstacle problem. From \Cref{thm:continuity and harmonicity} we know that $u$ is harmonic, therefore smooth, away from the contact set $\Lambda(u)$. In particular, this includes the right half $B_1'\setminus F$ of the horizontal line in $B_1$. For any $x_0 \in F$ that is neither $(0, 0)$ nor $(-1, 0)$, $u$ is a local solution to the zero thin obstacle problem. More precisely, there exists $\delta>0$ such that $u$ solves the $\mathcal K_{u, 0, B_{\delta}'(x_0)}(B_\delta(x_0))$-obstacle problem, where $B_{\delta}'(x_0) = \{x\in B_{\delta}(x_0) \mid x_2 = 0\}$ as always. Then, by known results for the thin obstacle problem or Signorini problem (see for example \cite{AC04 ,Sal12, PSU12}), $u$ is Lipschitz in an open neighbourhood along $x_0$ and $C^{1, 1\slash 2}$-regular from both sides up to $x_0$.

Hence, it suffices to study the regularity at $x_0 = 0$. In \Cref{thm:optimal holder reg half line} we prove the $C^{1\slash 2}$-regularity of $u$ in $x_0 = 0$. The proof of \Cref{thm:optimal holder reg half line} is based on the approach developed in \cite{CSS08, ACS08}, which is also available for obstacle problems with respect to the fractional Laplacian. The main strategy is to use Almgren's frequency function, \cite{Alm00},
\begin{equation*}
	N(r, u) \coloneqq  \frac{r\int_{B_r}|\nabla u|^2 \, dx }{\int_{\partial B_r}u^2 \, d\sigma }.
\end{equation*}
In order to show the monotonicity of $r\mapsto N(r, u)$ in \Cref{thm-almgren}, we first show a variant of Rellich's formula, \cite{Rel40}, in \Cref{prop:rellich}, which in the case of a smooth function $\varphi$ reads
 \begin{equation}\label{eq:rellich classic}
 	\int_{\partial B_r} |\nabla \varphi|^2 \, d\sigma = \frac{n-2}{r} \int_{B_r} |\nabla \varphi|^2 \, dx + 2 \int_{\partial B_r} (\partial_{\nu} \varphi )^2 \, d\sigma  - \frac{2}{r}\int_{B_r} (x\cdot \nabla \varphi) \Delta \varphi \, dx.
 \end{equation}
This formula turned out to be the main difficulty in the proof of the optimal regularity. In fact, it is the main reason that we are restricted to the two dimensional setting here, see \Cref{rem:rellich}. In the zero thin obstacle problem case, \eqref{eq:rellich classic} can be proved as follows: First, one shows the $C^{1, \alpha}$-regularity of $u$ for some $\alpha > 0$. Then, one uses that $\Delta u$ can be seen as a measure supported on the horizontal line in $B_1$. The $C^{1, \alpha}$-regularity of $u$ can then be used to show that \eqref{eq:rellich classic} holds for $\varphi = u$ as the derivative of $u$ is H\"{o}lder continuous. In particular, the last term on the right-hand side of \eqref{eq:rellich classic} vanishes as $\partial_{\nu} u \, \Delta u = 0$ in the sense of measures. See the proof of \cite[Lemma 1]{ACS08} or \cite[Theorem 2.5.1]{Sal12}.

Nevertheless, in the setting of \Cref{thm:optimal holder reg half line} we can only use the $C^{1, \alpha}$-regularity away from $0$. Thus, the idea is to apply Rellich's formula away from $0$ in  $B_r\setminus B_\varepsilon$ and then take the limit $\varepsilon\to 0$. The problematic term that arises is then of the form
\begin{equation}\label{eq:problematic term rellich proof}
	\varepsilon\int_{\partial B_\varepsilon} |\nabla u|^2 \, dx .
\end{equation}
To show that \eqref{eq:problematic term rellich proof} converges to $0$ for $\varepsilon\to 0$   we use the harmonicity away from $B_1'$ and the $C^\alpha$-regularity from \Cref{thm:preliminary holder reg} of $u$ to get the estimate $|\nabla u| \leq C \varepsilon^{\alpha -1}$ on $\partial B_\varepsilon$.

Then, with the help of Almgren's frequency function we can do a blowup analysis in \Cref{lem-classification} which proves the desired $C^{1\slash 2}$-regularity in \Cref{thm:optimal holder reg half line}. In particular, in \Cref{lem-classification} we show that any blowup $u_0$ of $u$ at $x_0 = 0 \in \Lambda(u)$ can be written in polar coordinates as
\begin{equation*}
    u_0(r, \theta) = r^{\kappa} \cos(\kappa\theta) \quad\text{for some $\kappa \in 2\mathbb N$ or $\kappa \in \mathbb N_0 + \frac 12$,}
\end{equation*}
up to a multiplicative constant.

\subsection{Preliminary gradient estimates}
We start with two simple estimates to control the gradient of harmonic functions and solutions to thin obstacle problems. Both of them hold for any dimension.
\begin{lemma}\label{lem:grad estimate harmonic}
	Let $n\geq 2$, $x_0\in \mathbb R^n$, and $r>0$. If $u$ is harmonic in $B_r(x_0)$, then
	\begin{equation*}
		|\nabla u(x_0)| \leq \frac{2\sqrt n}{r}\sup_{\partial B_{r\slash 2}(x_0)} |u|
	\end{equation*}
\end{lemma}
\begin{proof}
	Since $u$ is harmonic in $B_r(x_0)$, $\partial_i u$ is also harmonic. Thus, by the mean value property and integration by parts we have
	\begin{equation*}
		\partial_i u(x_0) = \dashint_{B_{r\slash 2}(x_0)} \partial_i u \, dx = \frac{2^n}{|B_1| r^n } \int_{\partial B_{r\slash 2}(x_0)} u \frac{x_i}{|x|} \, d\sigma,
	\end{equation*}
	which implies the statement.
\end{proof}

\begin{lemma}\label{lem:grad estimate zero thin}
	Let $B_1 \subset \mathbb R^n$ with $n\geq 2$, $u\in H^1(B_1)$, and $x\in B_1'$. If there exists $R>0$ such that $u$ is a solution to the zero thin obstacle problem in $B_R(x)$, then
	\begin{equation*}
		\sup_{B^{^\pm}_{r\slash 2}(x)} |\nabla u| \leq \frac{C}{r}\sup_{B_{3r/4}(x)} |u| \quad\text{for any $r\in (0, R]$},
	\end{equation*}
	for some constant $C>0$ depending only on $n$ and $R$.
\end{lemma}

\begin{proof}
	If $v\in H^1(B_1)$ is a solution to the zero thin obstacle problem, then $v\in C^{1, \alpha}(B_{3\slash 4}^+ )$, $\alpha >0$, with bound
	\begin{equation*}
		\norm{v}_{C^{1, \alpha}(B_{3\slash 4}^+ )} \leq C \norm{v}_{L^2(B_{3\slash 4})} \leq C \sup_{B_{3\slash 4}} |v|.
	\end{equation*}
	See for example \cite[Theorem 1]{Caf79}. Since $u$ is a local solution, by translation and scaling
	\begin{equation*}
		v(y)\coloneqq  u(ry + x) \in H^1(B_1)
	\end{equation*}
	is a solution to the zero thin obstacle problem. Thus, the statement follows from the relation
	\begin{equation*}
		|\nabla v(y)| = r|\nabla u(ry +x)|.
	\end{equation*}
\end{proof}

\subsection{Rellich-type formula}
\begin{proposition}[Rellich-type formula]\label{prop:rellich}
	Let $n =2$, $F= \{(x_1, 0) \mid x_1 \leq 0 \}$ and assume that $u\in H^1(B_1)$ is a solution to the $\mathcal K_{u, 0, F}(B_1)$-obstacle problem. Then,
	\begin{equation}\label{eq:rellich}
	 \int_{\partial B_r} |\nabla u|^2 \, d\sigma = 2\int_{\partial B_r} (\partial_{\nu}u)^2 \, d\sigma \quad\text{for every $0<r<1$}.
	\end{equation}
\end{proposition}
\begin{proof}
	We start as in the proof of \cite[(6)]{ACS08} or \cite[(13)]{Sal12}. Let
	\begin{equation*}
		V(x) \coloneqq  x|\nabla u|^2 - 2(x\cdot \nabla u) \nabla u \quad\text{for $x\in B_1 \setminus \Lambda(u)$.}
	\end{equation*}
	As long as we are away from $\Lambda(u)$, $V$ is $C^1$ as $u$ is harmonic and therefore smooth. In particular,
	\begin{equation*}
		\div V = (n-2)|\nabla u|^2 = 0 \quad\text{in $ B_1 \setminus \Lambda(u)$. }
	\end{equation*}
	As discussed before, for any $z\in F\setminus\{0\}$, $u$ is in $C^{1, 1\slash 2}(B_\delta(z)^+\cup B_\delta'(z))$ for some $\delta>0$. In particular, $V$ belongs to $C(B_\delta(z)^+\cup B_\delta'(z))$. For $0 < \delta < \varepsilon$ let
	\begin{equation*}
		F_\delta \coloneqq  \{x\in B_1 \mid \dist(x, F) < \delta \}.
	\end{equation*}
	Then, the divergence theorem implies
	\begin{equation}\label{eq:rellich proof div thm}
		0 = \int_{B_r \setminus (B_\varepsilon\cup F_\delta)} \div V \, dx = \int_{\partial (B_r \setminus (B_\varepsilon\cup F_\delta))} V\cdot \vec \nu  \, d\sigma. 
	\end{equation}
	We separate the surface integral into the following disjoint sets:
	\begin{equation*}
		\partial (B_r \setminus (B_\varepsilon\cup F_\delta)) = \partial B_r \setminus F_\delta \cup (\partial F_\delta \cap B_r) \setminus B_{\varepsilon} \cup \partial B_{\varepsilon}\setminus F_\delta \eqqcolon  A_1 \cup A_2\cup A_3. 
	\end{equation*}
	
    (i) On $A_1$ we have
	\begin{equation*}
		V\cdot \vec \nu = r|\nabla u|^2 - 2r(\partial_{\nu} u)^2,
	\end{equation*}
	which implies
	\begin{equation}\label{eq:rellich proof A1 int}
		\lim_{\varepsilon\to 0} \lim_{\delta \to 0} \int_{A_1} V\cdot \vec \nu \, d\sigma = r\int_{\partial B_r} |\nabla u|^2 - 2 (\partial_{\nu} u)^2 \, d\sigma.
	\end{equation}
	By comparing \eqref{eq:rellich proof div thm} and \eqref{eq:rellich proof A1 int} with \eqref{eq:rellich}, it remains to show that remaining integrals for $A_2$ and $A_3$ converge to 0.
	
	(ii) On $A_2$ let us consider $A_2^+ = A_2 \cap B_1^+$ and $A_2^- = A_2 \cap B_1^-$. Since every $x\in A_2^+$ is of the form $x = (x_1, \delta)$ and $\vec \nu(x) = (0, -1)$, we have $V\cdot \vec \nu = -\delta |\nabla u|^2 + 2(x\cdot \nabla u) \partial_{x_2} u$. Since $V$ is uniformly continuous up to $B_r'\setminus B_\varepsilon$, we have
	\begin{equation*}
		\lim_{\delta \to 0} V(x_1, \delta)\cdot \vec \nu = 2 \left( \begin{pmatrix}
			x_1 \\ 0
		\end{pmatrix}\cdot \nabla u(x_1, 0) \right)\partial_{x_2} u(x_1, 0).
	\end{equation*}
	Similarly, for $x = (x_1, -\delta)\in A_2^-$ we have 
	\begin{equation*}
		\lim_{\delta \to 0} V(x_1, -\delta)\cdot \vec \nu = -2\left( \begin{pmatrix}
			x_1 \\ 0
		\end{pmatrix}\cdot \nabla u(x_1, 0) \right)\partial_{x_2} u(x_1, 0).
	\end{equation*}
	Thus,
	\begin{equation*}
		\lim_{\delta\to 0}\int_{A_2} V\cdot \vec \nu \, d\sigma = \lim_{\delta\to 0}\int_{A_2^+} V\cdot \vec \nu \, d\sigma + \int_{A_2^-} V\cdot \vec \nu \, d\sigma = 0.
	\end{equation*}

	(iii) On $A_3 = \partial B_{\varepsilon}\setminus F_\delta$ we again use the regularity of $V$ on $\partial B_{\varepsilon}\cap \overline{F_\delta}$ to deduce
	\begin{equation*}
		\lim_{\delta \to 0 } \int_{A_3}  V\cdot \vec \nu \, d\sigma = \int_{\partial B_{\varepsilon}} V\cdot \vec \nu \, d\sigma.
	\end{equation*}
	On $\partial B_{\varepsilon}$ we have the bound
	\begin{equation*}
		|V\cdot \vec \nu| \leq 3\varepsilon|\nabla u|^2.
	\end{equation*}
	Thus, it is enough to show that
	\begin{equation}\label{eq:rellich proof A3 int after estimate}
		\lim_{\varepsilon\to 0} \varepsilon\int_{\partial B_\varepsilon} |\nabla u|^2 \, d\sigma = 0.
	\end{equation}
	Note, if there is $\alpha \in(0, 1)$ and a subset $D\subset \partial B_\varepsilon$ such that
	\begin{equation}\label{eq:rellich proof grad bound is enough}
		|\nabla u| \leq C \varepsilon^{\alpha -1} \quad\text{on $D$ for some constant $C$ independent of $\varepsilon$,}
	\end{equation}
    then we conclude
	\begin{equation*}
		\varepsilon\int_{D} |\nabla u|^2 \, d\sigma \leq \varepsilon^{1 + 2(\alpha -1)} \sigma(\partial B_\varepsilon) = \varepsilon^{2\alpha},
	\end{equation*}
	which converges to 0. Thus, if \eqref{eq:rellich proof grad bound is enough} holds for a finite partition of $\partial B_\varepsilon$, then \eqref{eq:rellich proof A3 int after estimate} follows.
	
	We show \eqref{eq:rellich proof grad bound is enough} in the two cases:
	\begin{equation*}
		D_1 \coloneqq  \{ x\in B_1 : x_{1} > 0 \text{ or } 2|x_2| \geq |x_{1}|\} \quad\text{and}\quad  	D_2 \coloneqq  \{ x\in B_1 : x_{1} < 0 \text{ and } 2|x_2| < |x_{1}| \}.
	\end{equation*}
	Assume $x\in D_1\cap \partial B_{\varepsilon}$. In the case $x_1 \geq 0$, $u$ is harmonic in $B_{\varepsilon}(x)$ and thus \Cref{lem:grad estimate harmonic} applied to $u - u(x)$ and additionally the H\"{o}lder estimate in \Cref{thm:preliminary holder reg} for some $\alpha \in(0, 1)$ imply
	\begin{equation*}
		|\nabla u(x)| \leq \frac{4}{\varepsilon}\sup_{\partial B_{\varepsilon\slash 2}(x)} |u| \leq C \varepsilon^{\alpha -1}
	\end{equation*}
	for some constant $C>0$ depending on $n$, $\alpha$ and $\norm{u}_{L^{\infty}(B_{3\slash 4})}$. In the case $x_1 < 0$, $u$ is harmonic in $B_{|x_2|}(x)$ as $\dist(x, F) = |x_2|$. As above, we have
	\begin{equation*}
		|\nabla u(x)|  \leq C |x_2|^{\alpha -1} \leq 5 C \varepsilon^{\alpha -1}
	\end{equation*}
	since $\varepsilon^2 = |x_1|^2 + |x_2|^2 \leq 5|x_2|^2$.
	
	Next, assume $x\in D_2 \cap \partial B_{\varepsilon}$. Set $x' \coloneqq  (x_1, 0)$. Note that
	\begin{equation*}
		|x - x'|^2 = |x_2|^2 < \frac 14  |x_1|^2.
	\end{equation*}
	If $B_{|x_1|}(x') \cap \Lambda(u) = \emptyset$, then $u$ is harmonic in $B_{|x_1|\slash 2}(x) \subset B_{|x_1|}(x')$. Thus, \Cref{lem:grad estimate harmonic}, \Cref{thm:preliminary holder reg}, and the estimate $\varepsilon^2 = |x_1|^2 + |x_2|^2 < \frac{5}{4}|x_1|^2$ imply
	\begin{equation*}
		|\nabla u(x)| \leq \frac{C}{|x_1|}\sup_{\partial B_{|x_1|\slash 4}(x)}|u-u(x)| \leq C \varepsilon^{\alpha -1}
	\end{equation*}
	for some constant $C>0$. In the case $B_{|x_1|}(x') \cap \Lambda(u) \neq \emptyset$, fix some $z\in B_{|x_1|}(x') \cap \Lambda(u)$. Then, we use that $u$ is a solution to the zero thin obstacle problem in $B_{|x_1|}(x')$ and $x\in B_{|x_1|\slash 2}(x')$. Thus, \Cref{lem:grad estimate zero thin} and \Cref{thm:preliminary holder reg} imply
	\begin{equation*}
		|\nabla u(x)| \leq \frac{C}{|x_1|}\sup_{B_{\frac{3}{4}|x_1| }(x')}|u| =\frac{C}{|x_1|}\sup_{B_{\frac{3}{4}|x_1| }(x')}|u - u(z)| \leq C|x_1|^{\alpha -1} \leq C \varepsilon^{\alpha -1}
	\end{equation*}
	for some constant $C>0$. This finishes the proof.
\end{proof}

\begin{remark}\label{rem:rellich}
	Let us comment on the dimensional constraint in \Cref{prop:rellich}.
	\begin{enumerate}
		\item Assume that $B_1 \subset \mathbb R^n$ for $n \geq 3$. If $F$ is the left half of the hyperplane, that is, $F= \{x \in B_1 \mid x_n = 0, x_{n-1} \leq 0\}$ as in \Cref{sec-almostoptimal}, then the critical set for the regularity is $\mathcal C \coloneqq  \{x \in B_1 \mid x_n = x_{n-1} = 0\}$. We can repeat the proof with $\mathcal C_\varepsilon = \{x \mid \dist(x, \mathcal C) < \varepsilon\}$ instead of $B_{\varepsilon}$ up to the estimates on $A_2$. On $A_3 = \partial\mathcal C_{\varepsilon}$ however, we only have the estimate
		\begin{equation*}
			|V\cdot \vec \nu| \leq 4(\varepsilon + |x'|)|\nabla u|^2.
		\end{equation*}
		Thus, we lose the linear factor of $\varepsilon$ in \eqref{eq:rellich proof A3 int after estimate}. Hence, it is not enough to show $|\nabla u| \leq C\varepsilon^{\alpha -1}$ for some $\alpha > 0$.
		\item If one replaces $B_1\subset \mathbb R^n$ by the cylinder $Z_1 \coloneqq  (-1, 1)^{n-2} \times B_1^{\mathbb R^2}(0) \subset \mathbb R^n$, $n\geq 3$, then one may choose
		\begin{equation*}
			V(x) \coloneqq  (x-x') |\nabla u|^2 - 2((x-x')\cdot \nabla u) \nabla u,
		\end{equation*}
		where $x' = (x_1, \ldots, x_{n-2}, 0, 0) \in Z_1$. Then, for $F= \{x \in Z_1 \mid x_n = 0, x_{n-1} \leq 0\}$ the proof in this case is analogous to \Cref{prop:rellich}.
	\end{enumerate}
	
\end{remark}

\subsection{Almgren's frequency function}

We are going to utilize \emph{Almgren's frequency function} as in \cite{ACS08}, see also \cite{PSU12, Sal12}. 
\begin{theorem}\label{thm-almgren}
	Let $n =2$, $F= \{(x_1, 0) \mid x_1 \leq 0 \}$, and assume that $u\in H^1(B_R)$, $R>0$, is a solution to the $\mathcal K_{u, 0, F}(B_R)$-obstacle problem. Then,
	\begin{equation*}
		N(r) = N(r, u) \coloneqq  \frac{r\int_{B_r}|\nabla u|^2 \, dx }{\int_{\partial B_r} u^2 \, d\sigma }
	\end{equation*}
	is nondecreasing for $0<r<R$. Moreover, $N(r, u)  \equiv \kappa \in \mathbb R_+$ for $0<r<R$ if and only if $u$ is homogenous of degree $\kappa$ in $B_R$, that is,
	\begin{equation*}
		x\cdot \nabla u - \kappa u = 0 \quad\text{in $B_R$.}
	\end{equation*} 
\end{theorem}
\begin{proof}
	We follow the proof of \cite[Lemma 1]{ACS08}. Let
	\begin{equation*}
		D(r) \coloneqq  \int_{B_r}|\nabla u|^2 \, dx \quad\text{and}\quad H(r) \coloneqq  \int_{\partial B_r} u^2 \, d\sigma.
	\end{equation*}
	Since $u$ is a local solution to the zero thin obstacle problem at every $x_0 \in F\setminus\{0\}$, $u$ is $C^{1, 1\slash 2}$ at every such point seen from $B_R^+$ and $B_R^-$. Thus, we have
	\begin{equation*}
		D'(r) = \int_{\partial B_r} |\nabla u|^2 \, d\sigma \quad\text{and}\quad H'(r) = \frac{n-1}{r} H(r) + 2 \int_{\partial B_r} u \partial_{\nu} u \, d\sigma,
	\end{equation*}
	which immediately imply that
	\begin{equation}\label{eq:ln N prime}
		\frac{d}{dr}\log N(r) = \frac{2-n}{r} + \frac{\int_{\partial B_r}|\nabla u|^2 \, d\sigma }{\int_{B_r}|\nabla u|^2 } - 2 \frac{\int_{\partial B_r}u\partial_{\nu} u  \, d\sigma }{\int_{\partial B_r}u^2 \, d\sigma }.
	\end{equation}

	The monotonicity of $N$ follows from $\frac{d}{dr}\log N \geq 0$. An application of Rellich's formula \eqref{eq:rellich} to rewrite $D'(r)$ and \eqref{eq:integration by parts u} to rewrite $D(r)$ in \eqref{eq:ln N prime} implies
	\begin{equation*}
		\frac{d}{dr}\log N(r) = 2\left(\frac{\int_{\partial B_r}\partial_{\nu} u^2 \, d\sigma }{\int_{\partial B_r} u\partial_{\nu} u \, d\sigma } -\frac{\int_{\partial B_r}u\partial_{\nu} u  \, d\sigma }{\int_{\partial B_r}u^2 \, d\sigma }\right).
	\end{equation*}
	We see that $\frac{d}{dr}\log N(r) \geq 0$ is equivalent to the Cauchy--Schwarz inequality
	\begin{equation}\label{eq:almgren cauchy schwarz}
		\int_{\partial B_r} u\partial_{\nu} u \, d\sigma \leq \left(\int_{\partial B_r} u^2 \, d\sigma \right)^{\frac 12}\left(\int_{\partial B_r} \partial_{\nu} u^2 \, d\sigma \right)^{\frac 12}.
	\end{equation}
	Thus, $N$ is nondecreasing. Moreover, $N'(r) = 0$ if and only if equality holds in \eqref{eq:almgren cauchy schwarz}. Hence, Cauchy-Schwarz implies that $N$ is constant if and only if $u$ and $\partial_{\nu} u$ are linearly dependent for every $r$.	 
	
	Let us now assume that $N(r) = \kappa \in \mathbb R_+$ for every $r$. Thus, there exists $f :(0, R) \to \mathbb R$ such that $\partial_{\nu} u(x) = f(|x|) u(x)$. A simple calculation using  \eqref{eq:integration by parts u} shows that $\kappa = rf(r)$, which implies
	\begin{equation*}
		0 = r\partial_{\nu} u(x) - rf(r)u(x) = \nabla u \cdot x - \kappa u(x).
	\end{equation*}
	In the contrary case, we get that $\partial_{\nu} u = \frac{\kappa}{r} u$ which implies that $u$ solves the ordinary differential equation
	\begin{equation*}
		\frac{d}{dr} u(x) = \nabla u(x) \cdot \frac{x}{r} = \frac{\kappa}{r} u(x).
	\end{equation*}
	This further implies that $u$ is of the form
	\begin{equation*}
		u(x) = r^{\kappa} g(\theta) \quad\text{for any $x$ with polar representation $(r, \theta)$},
	\end{equation*}
	for a function $g : \partial B_1 \to \mathbb R$. This implies that $N \equiv \kappa$.
 \end{proof}

\begin{remark}
	The dimensional constraint in \Cref{thm-almgren} is only due to Rellich's formula in \Cref{prop:rellich}.
\end{remark}
The main application of \Cref{thm-almgren} is the following growth estimate for solutions at the origin, which is the same as \cite[Lemma 9.14]{PSU12}:
\begin{lemma}[Growth estimate]\label{lem:growth estimate}
	Let $n =2$, $F= \{(x_1, 0) \mid x_1 \leq 0 \}$, and assume that $u\in H^1(B_1)$ is a solution to the $\mathcal K_{u, 0, F}(B_1)$-obstacle problem with $0\in \Lambda(u)$. Let $\kappa \coloneqq  N(0+, u)$. Then, there exists a constant $C = C(n, \kappa, \norm{u}_{L^\infty(B_{3\slash 4})})$ such that 
	\begin{equation*}
		\sup_{B_r} |u| \leq C r^{\kappa} \quad\text{for every $0<r<1\slash 2$.}
	\end{equation*}
\end{lemma}
Hence, to prove the optimal regularity it remains to show a lower bound on $\kappa = N(0+, u)$.
\subsection{Blowup analysis and optimal regularity}
We now consider rescalings of $u_r$, $r>0$, of $u$  given by 
\begin{equation}\label{eq:rescalings}
	u_r(x)\coloneqq \frac{u(rx)}{\left(r^{1-n}\int_{\partial B_r}u^2\right)^{1/2}} .
\end{equation}
Note that $u_r$ is normalized in the sense of
\begin{equation}\label{eq:blowup uni bound L2 boundary}
	\|u_r\|_{L^2(\partial B_1)}=1.
\end{equation}
Let us now assume that $u\in H^1(B_1)$ is a solution to the $\mathcal K_{u, 0, F}(B_1)$-obstacle problem for $F=\{(x_1, 0) \mid x_1 \leq 0 \}$ and additionally assume that $0\in \Lambda(u)$. Then, we investigate the blowups of $u$ at the origin, that is, the limits of $u_r$ for a subsequence of $r\searrow 0$. The existence of a blowup is guaranteed by the following result.
\begin{proposition}\label{prop:uni bound blowup}
	Let $n =2$, $F= \{(x_1, 0) \mid x_1 \leq 0 \}$, and assume that $u\in H^1(B_1)$ is a solution to the $\mathcal K_{u, 0, F}(B_1)$-obstacle problem. Then, the family $\{u_r\}_{r>0}$ of rescalings \eqref{eq:rescalings} is uniformly bounded in $H^1\cap C^{\alpha}_{\mathrm{loc}}$, for some $\alpha>0$, in the following sense: For any $r_0>0$ and $0<Rr_0<1$ there exists a constant $C(r_0)>0$ depending only on $r_0$ such that
	\begin{equation*}
		\norm{u_r}_{H^1(B_R)} + \norm{u_r}_{C^{\alpha}(B_{R\slash 2})} \leq C(r_0) \quad\text{for every $r\geq r_0$}.
	\end{equation*}
\end{proposition}
The proof of \Cref{prop:uni bound blowup} is standard and follows mainly from \Cref{thm-almgren} and \Cref{thm:preliminary holder reg}. For a proof, see for example \cite{ACS08} or \cite[page 177]{PSU12}. For any $R>0$ and some (small) $\alpha >0$ we can use  \Cref{prop:uni bound blowup} (with double the radius and a slightly bigger H\"{o}lder exponent) to get a subsequence $u_j \coloneqq  u_{r_j}$ and a limit $u_0 \in H^1(B_R) \cap C^{\alpha}(B_R)$ such that
\begin{equation}\label{eq:blowup convergence}
	\begin{aligned}
		u_{j} &\to u_0 \quad \text{weakly in $H^1(B_R)$},\\
		u_{j} &\to u_0 \quad \text{in $L^2(\partial B_1)$},\\
		u_{j} &\to u_0 \quad \text{in $C^{\alpha}(B_R)$}.\\
	\end{aligned}
\end{equation}
Note, by extracting further subsequences of $\{r_j\}_j$ for an increasing sequence of $R$, we actually have $u_0\in H^1_{\mathrm{loc}}(\mathbb R^n) \cap C^{\alpha}_{\mathrm{loc}}(\mathbb R^n)$. Moreover, there exists a subsequence of $\{r_j\}_j$, which we keep denoting by $\{r_j\}_j$, such that the convergence in \eqref{eq:blowup convergence} holds for any $R>0$. We call $u_0$ the \emph{blowup (at $x_0  = 0$)} of $u$.

\begin{lemma}[Homogeneity of blowups]\label{lem-blowup}
	Let $n =2$, $F= \{(x_1, 0) \mid x_1 \leq 0 \}$, assume that $u\in H^1(B_1)$ is a solution to the $\mathcal K_{u, 0, F}(B_1)$-obstacle problem and let $u_0$ be a blowup of $u$. Then, $u_0$ is a non-zero \emph{global solution to the $F$-Signorini problem}, that is, $u_0$ solves the $\mathcal K_{u_0, 0, F}(B_R)$-obstacle problem for every $R>0$. Moreover, $u_0$ is homogenous of degree $\kappa = N(0+, u)$ in $B_R$ for every $R>0$.
\end{lemma}
\begin{proof}
	From the weak $H^1$-convergence of the $u_j$, it is clear that $u_0$ is harmonic in $\mathbb R^n\setminus F$ and superharmonic in $\mathbb R^n$. Moreover, by the $C^\alpha$-convergence of $u_j$, we have $u_0\geq 0$ on $F$ as clearly $u_r \geq 0$ on $F\cap B_R$. On the other hand, for any $x_0\in F\setminus \Lambda(u_0)$, the uniform convergence implies that $x_0\in F\setminus \Lambda(u_j)$ for all sufficiently large $j$. In particular, such $u_j$ are harmonic in $x_0$ and therefore $u_0$ is harmonic in $x_0$. Thus, $u_0$ satisfies the Euler-Lagrange equations \eqref{eq:euler lagrange} and hence  \Cref{lem:euler lagrange} implies that $u_0$ is a global solution.  In particular, $u_0$ is non-zero due to \eqref{eq:blowup uni bound L2 boundary}.
	
	As in the proof of \cite[Proposition 9.5]{PSU12} we conclude that $H(r, u_0) = \int_{\partial B_r} u_0^2 \, d\sigma > 0$ for any $r>0$. From the Gauss-Green formula in \Cref{lem:integration by parts u}, for every $r>0$ we have
	\begin{equation*}
		D(r, u_0) = \int_{B_r}|\nabla u_0|^2=\int_{\partial B_r}u_0 \partial_{\nu} u_0=\lim_{j \to \infty}\int_{\partial B_r}u_{j} \partial_{\nu} u_{j}=\lim_{j \to \infty} \int_{B_r}|\nabla u_{j}|^2.
	\end{equation*}
	Note, the limit holds due to the $C^{1,\alpha}$-regularity of $u_j$ away from $0$, for some $\alpha>0$, by being a (local) solution to a thin obstacle problem. Thus, together with the uniform convergence $u_j \to  u_0$ and the fact that $H(r, u_0) > 0$, we conclude for every $r>0$
	\begin{equation*}
		N(r, u_0)=\lim_{j \to \infty}N(r, u_{j})=\lim_{j \to \infty}N(rr_j, u)=\kappa.
	\end{equation*}
	Finally, \Cref{thm-almgren} implies the homogeneity of $u_0$.
\end{proof}

\begin{theorem}[Classification of blowups]\label{lem-classification}
	Let $n =2$, $F= \{(x_1, 0) \mid x_1 \leq 0 \}$, assume that $u\in H^1(B_1)$ is a symmetric solution to the $\mathcal K_{u, 0, F}(B_1)$-obstacle problem and let $u_0$ be a blowup of $u$ with homogeneity $\kappa \coloneqq  N(0+, u)$. In particular, the following values of $\kappa$ are possible:
	\begin{enumerate}
		\item In general, $\kappa \in \frac 12 \mathbb N$.
		\item If $0 \in \Lambda(u)$ is an isolated contact point, that is, there exists $\delta > 0$ such that $B_{\delta}(0)\cap \Lambda(u) = \{0\}$, then $\kappa \in 2\mathbb N$.
		\item If $0 \in \Lambda(u)$ is an accumulation point of $\Lambda(u)$, then $\kappa \in \mathbb N_0 + \frac 12$.
	\end{enumerate}
    In every of the previous cases, $u_0$ is of the form
		\begin{equation*}
			u_0(r, \theta) = \begin{cases}
				r^{\kappa}\cos(\kappa \theta) &\text{if  $\kappa\in 2\mathbb N$ or $\kappa \in 2\mathbb N_0 + \frac 32$}\\
				-r^{\kappa}\cos(\kappa \theta) &\text{if  $\kappa\in 2\mathbb N_0 + \frac 12$.}\\
			\end{cases}
		\end{equation*}
		up to a positive multiplicative constant.
\end{theorem}

\begin{proof}
	First, we determine the possible values of $\kappa=N(0+, u)$. Since $n=2$, we can compute several candidates for global solutions $u_0$ as in \cite[Page 178]{PSU12}. More precisely, by using polar coordinates and the homogeneity of $u_0$ we have
	\begin{equation}\label{eq:homogenous representation}
		u_0(r, \theta) = r^\kappa \Phi(\theta) \quad\text{for $(r, \theta) \in [0, +\infty)\times [-\pi, \pi)$,}
	\end{equation}
	for some $\Phi : [-\pi, \pi) \to \mathbb R$. Note, to derive \eqref{eq:homogenous representation} we use the homogeneity of $u_0$ from \Cref{lem-blowup} in the upper and lower half-space to deduce \eqref{eq:homogenous representation} there and conclude by continuity. In particular, since $u_0\in C^\alpha_{\mathrm{loc}}(\mathbb R^2)$ for some $\alpha > 0$, we immediately see that $\kappa \geq \alpha > 0$. Using the harmonicity of $u_0$ in the upper and lower half-space we see that $\Phi$ is of the form
	\begin{equation*}
		\Phi(\theta) = a\sin(\kappa|\theta|) + b \cos(\kappa\theta)
	\end{equation*}
	for some $a, b\in\mathbb R$.
	
	Next, using the symmetry one checks that $u_0$ satisfies the boundary conditions
	\begin{enumerate}[label=(\alph*)]
		\item $u_0 \partial_{\nu}u_0=0$ on $\{(x_1, 0) \mid x_1 < 0 \}$;
		\item $\partial_{\nu}u_0=0$ on the $\{(x_1, 0) \mid x_1 > 0 \}$. 
	\end{enumerate}
	We conclude that $\Phi$ is of the form
	\begin{equation*}
		\Phi(\theta)=b \cos (\kappa \theta).
	\end{equation*}
	In particular, the condition (a) implies that $\kappa \in \frac 12 \mathbb N$. Using the additional restrictions $u_0 \geq 0$ and $\partial_{\nu}u_0 \geq 0$ on $F$, there exists $m\in\mathbb N_0$ such that $\kappa$ can only be of the following form:
	\begin{equation*}
		\kappa=
		\begin{cases}
			2m, 2m+3/2 & \text{if $b > 0$}\\
			2m+1, 2m+1/2 & \text{if $b<0$}.
		\end{cases}
	\end{equation*}
	
	At last, we want to rule out the values $\kappa = 2m$ or $\kappa = 2m + 1$. Let us assume that $0\in\Lambda(u)$ is an isolated contact point, that is, there exists $\delta > 0$ such that $B_\delta(0) \cap \Lambda(u) = \{0\}$. In this case, $u$ is harmonic in the punctured disk $B_\delta(0) \setminus\{0\}$. By the removable singularity theorem for harmonic functions, we conclude that $u$ is also harmonic in $x = 0$. However, then $u_0$ is harmonic in $x=0$ and in particular smooth. This is only possible if $\kappa \in 2\mathbb N$.
	
	Next, assume that $0\in \Lambda(u)$ is an accumulation point of $\Lambda(u)$. Then, there exists a sequence of decreasing radii $r_{k} >0$ converging to $0$ such that $u(-r_k\slash 2, 0) = 0$ for every $k$. Note, by repeating the construction of a blowup based on the family $\{u_{r_k}\}_k$, there exists a subsequence, which will also be denoted by $\{r_k\}_k$, such that $u_{r_k}$ converges to a blowup $\widetilde u_0$ in the sense of \eqref{eq:blowup convergence}. However, from the previous calculation, we can see that $\widetilde u_0$ has to be of the same form as $u_0$, hence $\widetilde u_0 = \widetilde bu_0$, for some $\widetilde b\in \mathbb R$. Moreover, from the definition of the rescalings \eqref{eq:rescalings} we see that $u(-r_k\slash 2, 0) = 0$ implies that $u_{r_k}(-1\slash 2, 0) = 0$ and therefore $u_0(-1\slash 2, 0) = 0$. Thus, $\cos(\kappa\pi) = 0$, which only allows for $\kappa \in \mathbb N_0+ \frac 12$.
\end{proof}

\begin{theorem}[Optimal regularity]\label{thm:optimal holder reg half line}
	Let $n =2$, $F= \{(x_1, 0) \mid x_1 \leq 0 \}$, and assume that $u\in H^1(B_1)$ is a solution to the $\mathcal K_{u, 0, F}(B_1)$-obstacle problem. Then, $u\in C^{1\slash 2}_{\mathrm{loc}}(B_1)$.
\end{theorem}
\begin{proof}
	Clearly, if $0\notin \Lambda(u)$, then $u$ is smooth in a small neighbourhood of $x_0 = 0$ and therefore there is nothing to do. Assume $0\in\Lambda(u)$. In the case when $u$ is symmetric, the result follows from the growth estimate \Cref{lem:growth estimate} and the classification of blowups \Cref{lem-classification}, which implies $\kappa = N(0+, u) \geq 1\slash 2$. If $u$ is not symmetric, then consider $w\in H^1(B_1)$ as the weak solution to the boundary value problem
	\begin{equation*}
		\left\{\begin{aligned}
			-\Delta w &= 0 &&\text{in $B_1$}\\
			w &= g &&\text{on $\partial B_1$,}
		\end{aligned}\right.
	\end{equation*}
	where $g(x, y) = \left(u(x, y) - u(x, -y) \right)\slash 2$. It is easy to check that $v \coloneqq  u-w$ is a symmetric solution to the $\mathcal K_{v, 0, F}(B_1)$-obstacle problem. Thus, $v\in C^{1\slash 2}_{\mathrm{loc}}(B_1)$ and since $w$ is smooth, we have $u\in C^{1\slash 2}_{\mathrm{loc}}(B_{1})$.
\end{proof}

\begin{remark}\label{rem:acf non zero}
We expect that one can extend the approach in \Cref{sec:opt holder reg} to include sufficiently regular obstacle data $\psi$ following the ideas in \cite{CSS08}. Let us give some instructions. For this, fix an obstacle $\psi$ with extension $\Psi\in H^1(B_1)$, let $u\in H^1(B_1)$ solve the $\mathcal K_{u, \psi, F}(B_1)$-obstacle problem, and set $v\coloneqq u - \Psi$. Then, $v$ satisfies the inhomogeneous zero-obstacle problem \eqref{eq:euler lagrange inhom} for $f\coloneqq \Delta\Psi$. First, one shows Rellich's formula for $v$:
\begin{equation*}
	\int_{\partial B_r} |\nabla v|^2 \, d\sigma =  2 \int_{\partial B_r} (\partial_{\nu} v )^2 \, d\sigma  + \frac{2}{r}\int_{B_r} (x\cdot \nabla v) f \, dx.
\end{equation*}
This can then be used to prove the monotonicity of the modified Almgren's frequency
\begin{equation*}
	N^f(r) \coloneqq (1+C_0r) \frac{r\int_{B_r}|\nabla v|^2 - fv \, dx }{\int_{\partial B_r} u^2 \, d\sigma}
\end{equation*}
for a sufficiently large constant $C_0>0$, similar to \cite[Theorem 3.1 and Appendix]{CSS08}. Then, as in \cite[Section 6]{CSS08}, one constructs a blowup $v_0$ from the sequence $v_r$, defined as in \eqref{eq:rescalings}, using $N^f$. Note that 
\begin{equation*}
    -\Delta v_r(x) = \frac{r^2 f(rx)}{\dashint_{\partial B_r} v^2 \, d\sigma},
\end{equation*}
which converges to 0 for $r\to 0$, at least for  H\"{o}lder regular $\Psi$ and bounded $f = \Delta \Psi$. Note, H\"{o}lder regularity of $\Psi$ implies H\"{o}lder regularity of $v$ due to \Cref{thm:preliminary holder reg}. Thus, the blowup $v_0$ is a global solution to the homogeneous zero-obstacle Signorini problem as in \Cref{lem-blowup}. Hence, the homogeneity of the blowup $\kappa := N^f(0+)$ is bounded from below by $N(r, v_0) \geq 1\slash 2$. One concludes the optimal regularity similar to \cite[Theorem 6.7]{CSS08}.
\end{remark}

\section{Application: Obstacles with jump-type discontinuities}\label{sec:application}
Let $\Omega\subset \mathbb R^n$ be open and $F\subset \Omega$ be relatively closed. Moreover, assume that there is a finite partition of $F$, that is, there is a positive integer $N$, relatively closed subsets $F_i \subset \Omega$, $i=1, \ldots, N$, such that $F = \bigcup_{i=1}^N F_i$, and $F_i\cap F_j$ is at most a set of codimension 2 for each $i \neq j$. For the obstacle, let $\psi_i \in C(F_i)$ and consider $\psi: F\to\mathbb R$ as the glued obstacle:
\begin{equation}\label{eq:glued obstacle}
	\psi(x) = \sup_{j : x \in F_i\cap F_j} \psi_j(x) \quad\text{if $x\in F_i$.}
\end{equation}
For any $x\in F$ let us denote $J_x \coloneqq \{j : x\in F_j\}$.
\begin{theorem}\label{thm:disc obstacle}
	Let $\Omega$, $F$ and $\psi$ be as described above. Assume that $u\in H^1(\Omega)$ is a solution to the $\mathcal K_{u, \psi, F}(\Omega)$-obstacle problem. Fix some $x_0\in F$. If there exists $i$ and $\delta > 0$ such that $x_0 \in F_i$ and
	\begin{equation}\label{eq:jump type disc}
		\psi_i(x_0) > \sup_{j\in J_{x_0}\setminus\{i\}} \sup_{F_j\cap B_{\delta}(x_0)} \psi_j,
	\end{equation}
	then there exists $0<\delta'\leq \delta$ such that $u$ is a solution to the $\mathcal K_{u, \psi_i, F_i}(B_{\delta'}(x_0))$-obstacle problem.
\end{theorem}
\begin{proof}
	By definition of $\psi$ in \eqref{eq:glued obstacle} it is clear that $\psi$ is upper semi-continuous on $F$. Thus, \Cref{thm:continuity and harmonicity} implies that $u$ is continuous in $\Omega$. Then, the obstacle condition $u\geq \psi$, \eqref{eq:jump type disc}, and the continuity of $u$ and the $\psi_j$'s imply that there exists $0<\delta'\leq \delta$ such that 
	\begin{equation}\label{eq:u larger than disc obstacles}
		u > \psi_j \quad\text{in $F_j\cap B_{\delta'}$ for every $j\neq i$.}
	\end{equation}
	To see this, assume the contrary. Then, there exists a sequence $\delta_n$ converging to $0$, $j_n \in J_{x_0}\setminus \{i\}$, and $x_n \in F_{j_n}\cap B_{\delta_n}(x_0)$ such that $u(x_n) = \psi_j(x_n)$. However, as $x_n \to x_0$ and $u$ is continuous, we then have
	\begin{equation*}
		\lim_{n\to\infty} \psi_{j_n}(x_n) = \lim_{n\to\infty} u(x_n) = u(x_0) \geq \psi_i(x_0),
	\end{equation*}
	which contradicts \eqref{eq:jump type disc}.
	
	Now, let us check the Euler-Lagrange equations \eqref{eq:euler lagrange} with respect to $B_{\delta'}(x_0)$, $F_i$, and $\psi_i$. Clearly, $u$ is superharmonic in $B_{\delta'}(x_0)$ by \Cref{prop:superharmonic}. Due to \eqref{eq:u larger than disc obstacles}, we have $\Lambda(u) \cap B_{\delta'} = \{x\in F_i \mid u(x) = \psi_i(x)\} \cap B_{\delta'}(x_0)$. Thus, $u$ is harmonic in $B_{\delta'}(x_0) \cap \{x\in F_i \mid u(x) = \psi_i(x)\}$ by \Cref{thm:continuity and harmonicity}. At last, $u \geq \psi_i$ on $F_i$ clearly holds. Hence, \Cref{lem:euler lagrange} implies that $u$ solves the $\mathcal K_{u, \psi_i, F_i}(B_{\delta'}(x_0))$-obstacle problem.
\end{proof}
Next, we present implications of \Cref{thm:disc obstacle} for thin obstacle problems within the settings of our regularity results for Signorini-type problems. The optimal regularity result in \Cref{thm:optimal holder reg half line} implies:
\begin{corollary}\label{cor:indicator function obstacle}
	Let $\Omega = B_1 \subset \mathbb R^2$, $F=B_1'$, and $\psi: B_1' \to \mathbb R$ be given by
	\begin{equation*}
		\psi(x_1, 0) \coloneqq  \mathds 1_{(-1, 0]}(x_1).
	\end{equation*}
	Then, any solution $u\in H^1(B_1)$ to the $\mathcal K_{u, \psi, B_1'}(B_1)$-obstacle problem is $C^{1\slash 2}$-regular at $x_0 = 0$.
\end{corollary}
\begin{proof}
	Clearly, every assumption in \Cref{thm:disc obstacle} is satisfied for obvious choices of $F_i$ and $\psi_i$. Thus, there is $\delta'>0$ such that $u$ solves the $\mathcal K_{u, 1, (-\delta', 0]}(B_{\delta'}(0))$-obstacle problem. It is easy to see that $v \coloneqq  u-1$ solves the $\mathcal K_{v, 0, (-\delta', 0]}(B_{\delta'}(0))$-obstacle problem. \Cref{thm:optimal holder reg half line} implies that $v$ is $C^{1\slash 2}_{\mathrm{loc}}(B_{\delta'}(0))$. Hence, $u$ is $C^{1\slash 2}$-regular at $x_0 = 0$.
\end{proof}
\begin{remark}
    As discussed in \Cref{rem:acf non zero}, we expect that a variant of the regularity result \Cref{thm:optimal holder reg half line} for non-zero obstacles hold. Hence, a variant of \Cref{cor:indicator function obstacle} for obstacles of the form
	\begin{equation*}
		\psi(x_1, 0) = \psi_1(x_1)\mathds 1_{(-1, 0]}(x_1) + \psi_2(x_1)\mathds 1_{(0, 1)}(x_1),
	\end{equation*}
	where $\psi_1, \psi_2$ are sufficiently regular and satisfy $\psi_1(0) > \psi_2(0)$, holds.
\end{remark}
The almost optimal regularity result in \Cref{thm:almost optimal non zero} implies:
\begin{corollary}
    Let $\Omega = B_1 \subset \mathbb R^n$, $n\geq 2$, and $F=\{(x', x_{n-1}, 0) \mid 0 \leq x_{n-1} < 1 \}$. Assume that $\psi : B_1' \to \mathbb R$ is of the form
    \begin{equation*}
        \psi(x', 0) = \psi_1(x')\mathds 1_{F}(x') + \psi_2(x')\mathds 1_{B_1'\setminus F}(x'),
    \end{equation*}
    for $\psi_1$, $\psi_2$ upper semi-continuous. If $\psi_1$ is in $C^{\beta_1}(F)$, for some $\beta_1\in(0, 1)$, has an extension as in \Cref{thm:almost optimal non zero}, and satisfies
    \begin{equation*}
        \psi_1(0) > \sup_{B_\delta\cap (B_1'\setminus F)} \psi_2
    \end{equation*}
    for some $\delta>0$, then any solution $u\in H^1(B_1)$ to the $\mathcal K_{u, \psi, B_1'}(B_1)$-obstacle problem is $C^{\beta}$ regular at $x_0 =0$, for any $\beta\in (0, 1\slash 2)$  such that $\beta\leq \beta_1$.
\end{corollary}
The optimal regularity result in \Cref{thm:optimalreg non zero} implies:
\begin{corollary}
    Let $\Omega = B_1 \subset \mathbb R^n$, $n\geq 2$. Assume that $\psi$ is given as in \eqref{eq:glued obstacle} for a finite partition $F_j$ of $B_1'$ and $\psi_j : F_j \to \mathbb R$ satisfy the assumptions in \Cref{thm:optimalreg non zero}. If there exists $i$ and $x_0\in F_i$ such that \eqref{eq:jump type disc} holds, then any solution $u\in H^1(B_1)$ to the $\mathcal K_{u, \psi, B_1'}(B_1)$-obstacle problem, whose contact set satisfies the capacity density condition \eqref{eq:CDC contact set} at $x_0 \in \Lambda(u)$, is $C^{1\slash 2}$-regular at $x_0$.
\end{corollary}

At last, let us comment further on the Lipschitz regularity result obtained in \cite{Kin71}.  For this, let us show the following:
\begin{lemma}\label{lem:kin71}
    Let $\Omega = B_1 \subset \mathbb R^2$, $F=\{(x_1, 0) \mid x_1 \leq 0\}$, and $\psi\in C^{1, \alpha}(F)$, $\alpha\in(0, 1)$, that is nonnegative and satisfies $\psi(0) = 0$. Then, there exists a $C^{1, \alpha}$-extension $\Psi : B_1' \to \mathbb R$ of $\psi$ such that any solution $u\in H^1_0(B_1)$ to the $\mathcal K_{0, \psi, F}(B_1)$-obstacle problem is a solution to the $\mathcal K_{0, \Psi, B_1'}(B_1)$-obstacle problem.
\end{lemma}
\begin{proof}
    Let us choose $\Psi$ by
    \begin{equation*}
        \Psi(x_1) \coloneqq \begin{cases}
            \psi(x_1) &\text{if $x_1\leq 0$}\\
            \psi'(0) x_1 &\text{if $x_1 > 0$}.
        \end{cases}
    \end{equation*}
    Clearly, $\Psi\in C^{1, \alpha}(B_1')$. Since $\psi$ is nonnegative and $u = 0$ on $\partial B_1$, the maximum principle implies that $u$ is nonnegative. Thus, $u \geq \Psi$ on $B_1'$. In particular
    \begin{equation*}
        \{x\in F \mid u(x) = \psi(x)\} = \{x\in B_1' \mid u(x) = \Psi(x)\}.
    \end{equation*}
    Hence, $u$ satisfies the Euler-Lagrange equations \eqref{eq:euler lagrange} with respect to $\Psi$ and therefore $u$ solves the $\mathcal K_{0, \Psi, B_1'}(B_1)$-obstacle problem.
\end{proof}
The setting considered in \Cref{lem:kin71} coincides with the one in \cite[Theorem 2]{Kin71},  where the author proved Lipschitz regularity of the solution $u$. However, \Cref{lem:kin71} shows that this solution is actually a solution to a thin obstacle problem. Thus, the combination of \Cref{lem:kin71} and the regularity result in \cite{Ric78} imply that the solution $u$ is Lipschitz in $B_1$ and is $C^{1, \beta}$, for $\beta = \min \{1\slash 2, \alpha\}$, in the upper and lower half ball up to the horizontal line.

Note, \Cref{lem:kin71} is based on the same observation as \Cref{thm:disc obstacle}, that is: Any change of the obstacle that happens below a certain threshold, which depends on the boundary data, does not change the solution of the associated obstacle problem.


\end{document}